\documentclass[oneside,english]{amsart}
\usepackage[T1]{fontenc}
\usepackage[latin9]{inputenc}
\usepackage{amssymb}

\makeatletter
\numberwithin{equation}{section} 
\numberwithin{figure}{section} 
  \theoremstyle{plain}
  \newtheorem{thm}{Theorem}[section]
  \theoremstyle{remark}
  \newtheorem*{acknowledgement*}{Acknowledgement}
  \theoremstyle{remark}
  \newtheorem*{rem*}{Remark}
  \theoremstyle{plain}
  \newtheorem{lem}[thm]{Lemma}
  \theoremstyle{definition}
  \newtheorem{defn}[thm]{Definition}
  \theoremstyle{plain}
  \newtheorem{prop}[thm]{Proposition}
  \theoremstyle{plain}
  \newtheorem{cor}[thm]{Corollary}

\usepackage{alex}
\usepackage{amscd}
\date{}
\newcommand*{\longhookrightarrow}{\ensuremath{\lhook\joinrel\relbar\joinrel\rightarrow}}

\usepackage{babel}
\makeatother

\begin{document}

\title{Unramified representations of reductive groups over finite rings}

\author{Alexander Stasinski}

\begin{abstract}
Lusztig has given a construction of certain representations of reductive
groups over finite local principal ideal rings of characteristic $p$,
extending the construction of Deligne and Lusztig of representations
of reductive groups over finite fields. We generalize Lusztig's results
to reductive groups over arbitrary finite local rings. This generalization
uses the Greenberg functor and the theory of group schemes over Artinian
local rings.
\end{abstract}

\address{DPMMS\\
University of Cambridge\\
Wilberforce Rd\\
Cambridge\\
CB3 0WB\\
U.~K.}

\email{a.stasinski@dpmms.cam.ac.uk}

\maketitle

\section*{Introduction}

In \cite{lusztig-preprint-published} Lusztig gave a construction
of certain representations of a reductive group over a finite ring
coming from the ring of integers in a local field of characteristic
$p>0$, modulo some power of its maximal ideal. Such rings can equivalently
be characterized as finite local principal ideal rings of characteristic
$p$. The construction, which is cohomological in nature and is a
generalization of the construction of Deligne and Lusztig \cite{delignelusztig},
attaches irreducible representations to certain characters of {}``maximal
tori''. It was stated in \cite{lusztig-preprint-published} that
the restriction on the ring is not essential, and that a similar method
applies in the case when the ring is a finite quotient of the ring
of integers of an arbitrary non-archimedean local field. Thus the
first natural problem is to realize the construction for arbitrary
finite local principal ideal rings. Unlike the characteristic $p$
case, it turns out that for arbitrary rings of this type it is no
longer possible to stay in the realm of algebraic groups over fields,
and instead the proper setting is that of group schemes over Artinian
local rings, and the theory of the Greenberg functor. Now this general
setting makes it clear that the construction does not have to be restricted
to principal ideal rings, but can in fact be carried out uniformly
for reductive groups over arbitrary finite local rings. Since any
finite (commutative) ring $R$ can be decomposed as a direct sum of
finite local rings $R\cong\bigoplus_{i}R_{i}$, it follows that if
$\mathbf{G}$ is an affine group scheme over $R$, then the group
of points $\mathbf{G}(R)$ can be written as a direct sum $\mathbf{G}(R)\cong\bigoplus_{i}\mathbf{G}(R_{i})$.
In the study of representations of $\mathbf{G}(R)$ it is therefore
enough to consider representations of the points of $\mathbf{G}$
over finite local rings.

In this paper we generalize Lusztig's construction to reductive group
schemes over arbitrary finite local rings. In particular, we thus
go beyond the original conception of rings of integers in local fields.
The representations we construct depend in a sense only on the structure
of the reductive group scheme over the residue field of the local
ring, and are essentially independent of the arithmetic of the ring
in question. This is a reason why we call these \emph{unramified}
representations. As Lusztig remarks in \cite{lusztig-preprint-published},
it seems likely that the representations we construct (in the principal
ideal ring case with $r\geq2$ and $\mathbf{G}$ split) coincide with
those constructed by G\'erardin \cite{Gerardin} by a non-cohomological
method. The latter are closely related to unramified maximal tori
and the unramified discrete series representations of $p$-adic groups,
and is another reason for our choice of terminology. There is also
some overlap with the representations constructed by Hill \cite{Hill_regular,Hill_semisimple_cuspidal}
in the case $\mathbf{G}=\text{GL}_{n}$, again by a non-cohomological
method.

After we have set up the framework of group schemes over local Artinian
rings and their associated algebraic groups, and proved several auxiliary
results, the proofs of the main theorems follow closely those of Lusztig.
We shall therefore give here a detailed comparison between the present
paper and the contents of \cite{lusztig-preprint-published}. 

The first section sets up some basic notation, introduces reductive
group schemes over Artinian local rings with residue field $\overline{\mathbb{F}}_{q}$,
the Greenberg functor, and the corresponding algebraic groups. For
the theory of group schemes we shall frequently refer to SGA~3 \cite{SGA3},
which seems to be the only complete reference covering what we need.
We sometimes also refer to Demazure's thesis \cite{Demazure}, which
is a convenient summary of many of the results we need from SGA~3.
For definitions and results concerning the Greenberg functor, we refer
to the original papers of Greenberg \cite{greenberg1,greenberg2}.
In \cite{lusztig-preprint-published}, group schemes could be bypassed
altogether because the base ring $\overline{\mathbb{F}}_{q}[[\epsilon]]/(\epsilon^{r})$
is an $\overline{\mathbb{F}}_{q}$-algebra, and so one can start with
an affine algebraic group $G_{1}$ over $\overline{\mathbb{F}}_{q}$,
and consider the group of points $G=G_{1}(\overline{\mathbb{F}}_{q}[[\epsilon]]/(\epsilon^{r}))$.
By using elementary considerations rather than the general formalism
of the Greenberg functor, one can then show that $G$ carries a structure
of affine algebraic group over $\overline{\mathbb{F}}_{q}$. Moreover,
there is a natural inclusion $G_{1}\subseteq G$, so the whole subgroup
structure of $G_{1}$ can be easily transferred to $G$. In the general
situation considered in the present paper, we are forced instead to
start with a reductive group scheme $\mathbf{G}$ over an Artinian
local ring $A$ with residue field $\overline{\mathbb{F}}_{q}$, and
work directly with the structure of $\mathbf{G}$. Throughout the
paper, we shall fix the ring $A$, write $\mathfrak{m}$ for its maximal
ideal, and denote by $r$ the smallest positive integer such that
$\mathfrak{m}^{r}=0$. We then simply write $G$ for the algebraic
group associated to $\mathbf{G}(A)$. One reason why this approach
is possible is that the theory of affine smooth group schemes over
strictly Henselian local rings in many ways resembles the classical
theory of algebraic groups over algebraically closed fields.

In the second section, Lemmas \ref{lem:Lusztig's 1.1}, \ref{lem:Lusztig's 1.3},
\ref{lem:Lusztig's 1.4}, and \ref{lem: induction-commutator} are
due to Lusztig, although of course stated here in our wider generality.
In order to clarify some steps of the proof of Lemma~\ref{lem:Lusztig's 1.3},
we have given a proof using the new Lemma~\ref{lem:Unique decomposition}.
Several results in \cite{lusztig-preprint-published} were stated
under the hypothesis $r\geq2$, and since the results were already
known for $r=1$, this is no loss of generality. Nevertheless, we
have removed this hypothesis in order to emphasize that the proofs
in fact work uniformly for all $r\geq1$, except for the proof of
Theorem~\ref{thm:Big-thm} where one has to separate the case $r=1$
(where regularity of characters is not required), from the case $r\geq2$
(where regularity is used, and the proof becomes longer). This also
affects Prop.~\ref{pro:scalar product formula} and Coroll.~\ref{cor:irrep & indepU},
which we have stated in a form that includes the case $r=1$.

The proofs of Lemmas~\ref{lem:commutator-Iwahori}, \ref{lem:Unique decomposition},
\ref{lem:regularity test} and \ref{lem:roots-commutator} are new,
although the results have to some degree been known earlier either
implicitly in unpublished form or in various special cases. In particular,
Lemma~\ref{lem:commutator-Iwahori} provides a commutator relation
and Iwahori decomposition for (certain) group schemes over local rings.
These results are well-known for certain classes of groups and rings,
but our results hold quite generally and are proved by arguments which
are more geometric than the classical group theoretic approaches.
Lemmas~\ref{lem:regularity test} and \ref{lem:roots-commutator}
were stated in \cite{lusztig-preprint-published} without proof, and
since our proofs are not obvious, we have included them here. The
proof of Lemma~\ref{lem:roots-commutator}(c) is especially long,
and is an example of the extra complications that appear in our general
setting compared to the case of rings of characteristic $p$, where
the proof can be reduced to the case of $\mbox{SL}_{2}$.

In the final section, we have collected all the main results, including
our version of Lusztig's Lemma~1.9 which we have here given the status
of Theorem since its proof is the longest and most difficult in the
whole paper and its consequences include the most important results.
The ideas of the proofs in this section are due to Lusztig, and our
presentation follows \cite{lusztig-preprint-published}, except with
regards to the use of the elements $\hat{w}$ (see below). We have
also added some references to various results used in the proofs,
and some clarifying remarks. We have included these reworkings of
Lusztig's proofs in order to get a complete and coherent exposition,
and we believe this to be a more satisfying solution (both logically
and from the reader's point of view) than if we had simply stated
the generalized main results and referred to proofs appearing in a
more special context.

If $\mathbf{T}$ and $\mathbf{T}'$ are two maximal tori in $\mathbf{G}$,
we shall denote the corresponding closed subgroups of $G$ by $T$
and $T'$, respectively. Reducing modulo $\mathfrak{m}$ we also get
the maximal tori $T_{1}$ and $T_{1}'$ in $G_{1}$. A remark applicable
to both of the last two sections, is that unlike the case where the
ring $A$ has characteristic $p$, in general we cannot directly transfer
elements of the transporter $N(T_{1},T_{1}')=\{n\in G_{1}\mid n^{-1}T_{1}n=T_{1}'\}$
to elements of $N(T,T')=\{n\in G\mid n^{-1}Tn=T'\}$. Instead we use
results from SGA~3 showing that the transporter (or normalizer) group
schemes of maximal tori are smooth, and from this we conclude that
the natural map $N(T,T')\rightarrow N(T_{1},T_{1})$ is surjective.
For any element $\dot{w}\in N(T_{1},T_{1}')$ we can thus work with
a lift $\hat{w}\in N(T,T')$ under this map. It turns out that the
ambiguity in the choice of lifts does not affect the results, and
this provides a sense in which the results only depend on structures
over the residue field.

\begin{acknowledgement*}
Parts of this work were carried out while the author was supported
respectively by EPSRC Grants GR/T21714/01 and EP/C527402. The author
wishes to thank A.-M.~Aubert, V.~Snaith, and S.~Stevens for their
interest in and support of this work, B.~Totaro for helpful discussions,
and U.~Onn for an invitation and opportunity to lecture on this material.
\end{acknowledgement*}

\section{Notation}

Throughout this paper a ring will always refer to a (unital, associative)
commutative ring. Let $A$ be an Artinian local ring with maximal
ideal $\mathfrak{m}$ and perfect residue field $k$. Let $r$ denote
the smallest positive integer such that $\mathfrak{m}^{r}=0$. Let
$X$ be a scheme of finite type over $A$ (as usual, we shall speak
of a scheme over the ring $A$ rather than over the scheme $\Spec A$).
Greenberg \cite{greenberg1,greenberg2} has defined a functor $\mathcal{F}_{A}$
from the category of schemes of finite type over $A$ to the category
of schemes of finite type over $k$, such that there exists a canonical
isomorphism $X(A)\cong(\mathcal{F}_{A}X)(k)$. It is shown in loc.~cit\emph{.}~that
the functor $\mathcal{F}_{A}$ preserves affine and separated schemes,
respectively. Furthermore, it maps group schemes over $A$ to group
schemes over $k$, schemes smooth over $A$ to schemes smooth over
$k$, and preserves subschemes (of any kind). If $X$ is smooth over
$A$ and $X\times_{A}k$ is reduced and irreducible, then $\mathcal{F}_{A}X$
is reduced and irreducible (\cite{greenberg2}, 2, Coroll.~2).

Suppose that $\mathbf{G}$ is an affine smooth group scheme over $A$.
Thus it is in particular of finite type over $A$. We take the residue
field $k$ to be an algebraic closure $\overline{\mathbb{F}}_{q}$
of the finite field $\mathbb{F}_{q}$ of characteristic $p$. For
any integer $r'$ such that $r\geq r'\geq0$ we define\[
G_{r'}=\mathcal{F}_{A}(\mathbf{G}\times_{A}A/\mathfrak{m}^{r'})(k).\]
Note that for $r'=0$ the ring $A/\mathfrak{m}^{r'}$ is the trivial
ring $\{0=1\}$, and so $G_{0}$ consists of exactly one point. In
particular, for $r'=r$ we write $G$ for the group $G_{r}=(\mathcal{F}_{A}\mathbf{G})(k)\cong\mathbf{G}(A)$.
In general, we shall write group schemes over $A$ in boldface type,
and the corresponding algebraic group over $k$ using the same letter
in normal type. By the results of Greenberg, each group $G_{r'}$
is the $k$-points of an affine smooth group scheme over $k$. It
is thus an affine algebraic group over $k$, connected if $\mathbf{G}\times k$
is. Since $\mathbf{G}$ is smooth it follows that the reduction map
$A\rightarrow A/\mathfrak{m}^{r'}$ induces a surjective homomorphism
$\phi_{r,r'}=\phi_{r'}:G\rightarrow G{}_{r'}$. The kernel of $\phi_{r'}$
is denoted by $G^{r'}$. We have \[
\{1\}=G^{r}\subseteq G^{r-1}\subseteq\dots\subseteq G^{1}\subseteq G^{0}=G.\]
Let $G^{r',*}=G^{r'}-G^{r'+1}$, for $r'<r$. We thus have a partition\[
G=G^{0,*}\sqcup G^{1,*}\sqcup\dots\sqcup G^{r-1,*}\sqcup\{1\}.\]
From now on, let $\mathbf{G}$ be a reductive group scheme over $A$
(cf.~\cite{Demazure}, 2.1 or \cite{SGA3}, XIX 2.7). This means
that $\mathbf{G}$ is an affine smooth group scheme over $A$, such
that its fibre $\mathbf{G}\times k$ is a connected reductive group
in the classical sense. We shall be interested in the situation where
$G$ is endowed with a Frobenius endomorphism $F:G\rightarrow G$,
which in the most general sense is just an endomorphism with finite
fixed point group $G^{F}$.

\begin{rem*}
We show how a situation as above typically arises. Let $A_{0}$ be
an arbitrary finite local ring. Then $A_{0}$ is obviously Artinian
with residue field $\mathbb{F}_{q}$, for some $q$. Let $\mathbb{\mathbf{G}}_{0}$
be a reductive group scheme over $A_{0}$. Then by results of Greenberg,
$A_{0}$ is an algebra over the ring of Witt vectors $W_{n}(\mathbb{F}_{q})$,
where $\chara A_{0}=p^{n+1}$. Let\[
A=A_{0}\otimes_{W_{n}(\mathbb{F}_{q})}W_{n}(\overline{\mathbb{F}}_{q}).\]
Then by \cite{greenberg1}, 1, Prop.~4, $A$ is a local Artinian
ring with residue field $\overline{\mathbb{F}}_{q}$. The algebra
$A$ carries an endomorphism $F$ induced by the Frobenius map of
$W_{n}(\overline{\mathbb{F}}_{q})$. If we now let $\mathbf{G}=\mathbf{G}_{0}\times_{A_{0}}A$,
then $G$ inherits a Frobenius endomorphism from the endomorphism
$F$ on $\mathbb{\mathbf{G}}(A)$ such that $G^{F}\cong\mathbf{G}(A)^{F}$.
Note however that not all Frobenius endomorphisms of $G$ are of this
form; there are also those that give rise to twisted groups.
\end{rem*}
Assume henceforth that $G$ is an algebraic group over $\overline{\mathbb{F}}_{q}$
obtained from a reductive group scheme $\mathbf{G}$ as above, and
provided with a Frobenius endomorphism $F:G\rightarrow G$. Maximal
tori and Borel subgroups exist (resp.~are conjugate) in $\mathbf{G}$
locally for the \'etale topology, that is, after a possible \'etale
extension of scalars (see \cite{Demazure}, 1.5, 3.3 and \cite{SGA3},
XXII 5.2.1-5.2.3, for these notions and facts). Since our base is
a local ring with algebraically closed residue field (i.e., strictly
Henselian), maximal tori exist in $\mathbf{G}$, and any maximal torus
of $\mathbf{G}$ is contained in a Borel subgroup $\mathbf{G}$. If
$\mathbf{T}$ is a maximal torus contained in a Borel subgroup $\mathbf{B}$
of $\mathbf{G}$, we have a semidirect product $\mathbf{B}=\mathbf{T}\mathbf{U}$,
where $\mathbf{U}$ is the unipotent radical of $\mathbf{B}$ (cf.~\cite{SGA3},
XXII 5.11.4). We then have the respective associated algebraic subgroups
$T,B,U$ of $G$, and a semidirect product $B=TU$. Note that for
$r\geq2$, $T$ will not be a maximal torus of $G$ in the sense of
algebraic groups; nor will $B$ be a Borel subgroup of $G$. 

Throughout this paper we fix a prime $l\neq p$. If $X$ is an algebraic
variety over $\overline{\mathbb{F}}_{q}$ we write $H_{c}^{j}(X)$
instead of $H_{c}^{j}(X,\overline{\mathbb{Q}}_{l})$. For a finite
group $\Gamma$, we write $\widehat{\Gamma}$ for the group of linear
characters $\Hom(\Gamma,\overline{\mathbb{Q}}_{l}\vphantom{\overline{\mathbb{Q}}}^{\hspace{-3pt}\times})$.

\section{\label{sec:Lemmas}Lemmas}

Let $\mathbf{T}$ be a maximal torus of $\mathbf{G}$, and let $\Phi=\Phi(\mathbf{G},\mathbf{T})$
be the set of roots of $\mathbf{G}$, relative to $\mathbf{T}$ (cf.~\cite{SGA3},
XIX 3.6). Thus $\Phi$ consists of elements in $\Hom_{A\text{-gr}}(\mathbf{T},(\mathbb{G}_{m})_{A})$
whose image in $\Hom_{k\text{-gr}}(\mathbf{T}\times k,(\mathbb{G}_{m})_{k})$
is a root in the usual sense. We write the groups of characters additively,
and denote by $\mathbf{U}_{\alpha}$ the root subgroup of $\mathbf{G}$
corresponding to the root $\alpha$. 

We shall consider the notion of \emph{splitting} (d\'eploiement)
of a reductive group $\mathbf{G}$ with respect to a maximal torus
$\mathbf{T}$ (cf.~\cite{SGA3}, XXII 1.13, or \cite{Demazure},
3.1), and we call $\mathbf{G}$ \emph{split} (d\'eploy\'e) with
respect to $\mathbf{T}$, if such a splitting exists. A torus over
a general base $\mathbf{S}$ called \emph{trivial }if it is diagonalizable,
that is, if it is isomorphic to some $(\mathbb{G}_{m}^{n})_{\mathbf{S}}$.

\begin{lem}
\label{lem:splitting}A reductive group $\mathbf{G}$ over a strictly
Henselian ring $A$ is split with respect to any of its maximal tori.
\end{lem}
\begin{proof}
Over a strictly Henselian base $A$, any property that holds locally
for the \'etale topology holds already over $A$. A maximal torus
of $\mathbf{G}$ is trivial locally for the \'etale topology (\cite{Demazure}
1.4.5), and is thus trivial over $A$. Since $A$ is local, the result
follows from \cite{SGA3}, XXII 2.2.
\end{proof}
The existence of splitting implies that the root data of $\mathbf{G}$
relative to $\mathbf{T}$ is canonically isomorphic to the root data
of $\mathbf{G}\times k$ relative to $\mathbf{T}\times k$ (cf.~\cite{SGA3}
XXII 1.15 b)). In particular, the map $\Hom_{A\text{-gr}}(\mathbf{T},(\mathbb{G}_{m})_{A})\rightarrow\Hom_{k\text{-gr}}(\mathbf{T}\times k,(\mathbb{G}_{m})_{k})$
is a bijection on the roots. As for algebraic groups over fields,
a choice of Borel subgroup $\mathbf{B}$ of $\mathbf{G}$ containing
$\mathbf{T}$ defines a set of positive roots $\Phi^{+}$, and the
splitting of $\mathbf{G}$ with respect to $\mathbf{T}$ implies that
for some fixed but arbitrary ordering of $\Phi^{+}$ we have\[
\mathbf{U}=\prod_{\alpha\in\Phi^{+}}\mathbf{U}_{\alpha}\]
(see \cite{Demazure}, 3.3.3). On the level of groups of points this
yields $U=\prod_{\alpha\in\Phi^{+}}U_{\alpha}$, where an element
of $U$ is expressed uniquely as a product of elements of the $U_{\alpha}$.

From now on, let $\mathbf{T}$ and $\mathbf{T}'$ be two maximal tori
of $\mathbf{G}$ such that the corresponding subgroups $T$ and $T'$
of $G$ are $F$-stable. Let $\mathbf{U}$ (resp.~$\mathbf{U'}$)
be the unipotent radical of a Borel subgroup of $\mathbf{G}$ that
contains $\mathbf{T}$ (resp.~$\mathbf{T}'$), and let $U$ and $U'$
be the corresponding subgroups of $G$. Note that $U$ and $U'$ are
not necessarily $F$-stable. 

Let $N(T_{1},T_{1}')=\{g\in G_{1}\mid g^{-1}T_{1}g=T_{1}'\}$. Then
$T_{1}$ acts on $N(T_{1},T_{1}')$ by left multiplication and $T_{1}'$
acts on $N(T_{1},T_{1}')$ by right multiplication. The orbits of
$T_{1}$ are in natural bijection with the orbits of $T_{1}'$. We
set $W(T_{1},T_{1}')=T_{1}\backslash N(T_{1},T_{1}')\cong N(T_{1},T_{1}')/T_{1}'$;
this is a finite set because if $a\in G_{1}$ is an element such that
${}^{a}T_{1}=T_{1}'$, then $g\mapsto ga$ gives a bijection between
$N(T_{1},T_{1}')$ and the normalizer $N_{G_{1}}(T_{1})=N(T_{1},T_{1})$,
and this induces a bijection between $W(T_{1},T_{1}')$ and the Weyl
group $W(T_{1})=W(T_{1},T_{1})$. For each $w\in W(T_{1},T_{1}')$
we choose a representative $\dot{w}\in N(T_{1},T_{1}')$. Since the
normalizer $N_{\mathbf{G}}(\mathbf{T})$ is smooth over $A$ (cf.~\cite{Demazure},
1.5.1), the map $\phi_{1}$ induces a surjection\[
N_{\mathbf{G}}(\mathbf{T})(A)\longrightarrow N_{\mathbf{G}}(\mathbf{T})(k).\]
By the definition of the normalizer group scheme and the fact that
$k$ is algebraically closed, we have $N_{\mathbf{G}}(\mathbf{T})(A)\subseteq N_{\mathbf{G}(A)}(\mathbf{T}(A))$
and $N_{\mathbf{G}}(\mathbf{T})(k)=N_{\mathbf{G}(k)}(\mathbf{T}(k))$
(cf.~\cite{Jantzen}, I 2.6). Thus $\phi_{1}$ also induces a surjection\[
N_{G}(T)\cong N_{\mathbf{G}(A)}(\mathbf{T}(A))\longrightarrow N_{\mathbf{G}(k)}(\mathbf{T}(k))=N_{G_{1}}(T_{1}).\]
Let $N(T,T')=\{g\in G\mid g^{-1}Tg=T'\}$. It follows from the conjugacy
of maximal tori (\cite{Demazure}, 1.5.3) that $T$ and $T'$ are
conjugate in $G$ by an element whose image in $G_{1}$ conjugates
$T_{1}$ to $T_{1}'$. Thus we have in the same way as above, a bijection
between $N(T,T')$ and $N_{G}(T)$, and hence a surjection $N(T,T')\rightarrow N(T_{1},T_{1}')$
(this also follows from the smoothness of transporters \cite{SGA3},
XXII 5.3.9). For each $\dot{w}\in N(T_{1},T_{1}')$ we can therefore
choose a lift $\hat{w}\in N(T,T')$, and throughout this paper we
shall work with a fixed set of lifts $\hat{w}$. As we shall see,
the main results are independent of the choice of these lifts.

Define the variety\[
\Sigma=\{(x,x',y)\in F(U)\times F(U')\times G\mid xF(y)=yx'\}.\]
The Bruhat decomposition in $G_{1}$ implies that there is a bijection
between double $B_{1}$-$B_{1}$ cosets indexed by $W(T_{1},T_{1}')$,
and double $B_{1}$-$B_{1}$ cosets indexed by $W(T_{1})$. Indeed,
if $w\in W(T_{1},T_{1}')$ and if $a\in G_{1}$ is such that ${}^{a}T_{1}=T_{1}'$,
then the map $B_{1}wB_{1}\mapsto B_{1}waB_{1}$ is injective since
if $B_{1}w'B_{1}$ has the same image as $B_{1}wB_{1}$, then $B_{1}waB_{1}=B_{1}w'aB_{1}$,
so by Bruhat decomposition, $\dot{w}a$ and $\dot{w}'a$ have the
same image in $W(T_{1})$, that is, $\dot{w}$ and $\dot{w}'$ have
the same image in $W(T_{1},T_{1}')$. We thus have $G_{1}=\bigsqcup_{w\in W(T_{1},T_{1}')}G_{1,w}$,
where $G_{1,w}=U_{1}T_{1}\dot{w}U_{1}'=U_{1}\dot{w}T_{1}'U_{1}'$.
Let $G_{w}$ be the inverse image of $G_{1,w}$ under $\phi_{1}:G\rightarrow G_{1}$
and let \[
\Sigma_{w}=\{(x,x',y)\in\Sigma\mid y\in G_{w}\}.\]
This defines a partition of $\Sigma$. The group $T^{F}\times T'^{F}$
acts on $\Sigma$ by $(t,t'):(x,x',y)\mapsto(txt^{-1},t'x't'^{-1},tyt'^{-1})$.
This restricts to an action of $T^{F}\times T'^{F}$ on $\Sigma_{w}$
for any $w\in W(T_{1},T_{1}')$.

If $\theta\in\widehat{T^{F}}$, $\theta'\in\widehat{T'^{F}}$, and
$M$ is a $T^{F}\times T'^{F}$-module, we shall write $M_{\theta^{-1},\theta}$
for the subspace of $M$ on which $T^{F}\times T'^{F}$ acts according
to $\theta^{-1}\boxtimes\theta'$, that is, \[
M_{\theta^{-1},\theta'}=\{m\in M\mid(t,t')m=\theta^{-1}(t)\theta'(t')m,\ \forall\,(t,t')\in T^{F}\times T'^{F}\}.\]
\renewcommand{\labelenumi}{(\alph{enumi})}

\begin{lem}
\label{lem:commutator-Iwahori}Let $\mathbf{G}$ be an affine group
scheme over a local ring $A$ with maximal ideal $\mathfrak{m}$.
For $i\geq0$, write $G=\mathbf{G}(A)$, $G_{i}=\mathbf{G}(A/\mathfrak{m}^{i})$,
and $G^{i}=\Ker(G\rightarrow G_{i})$. Then the following holds:
\begin{enumerate}
\item For any integers $i,j\geq0$ we have the commutator relation $[G^{i},G^{j}]\subseteq G^{i+j}$.
\item (Iwahori decomposition) Assume in addition that $\mathbf{G}$ is reductive
and split over $A$, with respect to a maximal torus $\mathbf{T}$.
Let $\mathbf{T}$ be contained in a Borel subgroup with unipotent
radical $\mathbf{U}$, and let $\mathbf{U}^{-}$ be the unipotent
radical of a Borel subgroup of $\mathbf{G}$ containing $\mathbf{T}$,
such that $\mathbf{U}\cap\mathbf{U}^{-}=\{1\}$. Let $T$, $U$, $U^{-}$
be the respective groups of $A$-points, and let $T^{1}$, $U^{1}$
and $(U^{-})^{1}$ be the respective kernels. Then we have\[
G^{1}=(U^{-})^{1}T^{1}U^{1},\]
and each element $g\in G^{1}$ decomposes uniquely as $g=u^{-}tu$,
where $u^{-}\in(U^{-})^{1}$, $t\in T^{1}$, and $u\in U^{1}$. $ $
\end{enumerate}
\end{lem}
\begin{proof}
We prove (a) using a Hopf algebra approach. Let $A[\mathbf{G}]$ be
the affine algebra of $\mathbf{G}$; thus $A[\mathbf{G}]$ is a commutative
Hopf algebra over $A$. Let $\Delta:A[\mathbf{G}]\rightarrow A[\mathbf{G}]\otimes A[\mathbf{G}]$
and $\epsilon:A[\mathbf{G}]\rightarrow A$ denote its coproduct and
counit, respectively. Let $I=\Ker\epsilon$ be the augmentation ideal.
If $\alpha:A\rightarrow R$ is an $A$-algebra, then the identity
element of the group $G(R)=\Hom(A[\mathbf{G}],R)$ is given by $\alpha\circ\epsilon$.
For any $i\geq0$, the reduction map $\phi_{i}:G=\Hom(A[\mathbf{G}],A)\rightarrow\Hom(A[\mathbf{G}],A/\mathfrak{m}^{i})=G_{i}$
sends any $g\in G$ to $\phi_{i}\circ g$. Now let $g\in G^{i}$ and
$h\in G^{j}$, for some integers $i,j\geq0$. Then \[
\phi_{i}\circ g=\phi_{i}\circ\epsilon,\quad\text{and}\quad\phi_{j}\circ g=\phi_{j}\circ\epsilon,\]
respectively (recall that $\phi_{i}$ denotes both the map $G\rightarrow G_{i}$
and map $A\rightarrow A/\mathfrak{m}^{i}$). Thus $\phi_{i}(g(I))=0$,
that is, $g(I)\subseteq\mathfrak{m}^{i}$. Similarly, we have $h(I)\subseteq\mathfrak{m}^{j}$.
Since $a\mapsto a\cdot1:A\rightarrow A[\mathbf{G}]$ is a section
of $\epsilon$, we have $A[\mathbf{G}]=A\cdot1\oplus I$, as $A$-modules.
This implies that $A[\mathbf{G}]\otimes A[\mathbf{G}]=A(1\otimes1)\oplus(A\otimes I)\oplus(I\otimes A)\oplus(I\otimes I).$
Let $x\in I$, and write $\Delta(x)=a_{1}(1\otimes1)+a_{2}\otimes y_{1}+y_{2}\otimes a_{3}+y_{3}\otimes y_{4}$,
where $a_{k}\in A$, $y_{k}\in I$. The Hopf algebra axiom $(\epsilon\circ\text{id})\circ\Delta=\text{id}=(\text{id}\circ\epsilon)\circ\Delta$,
implies that $a_{1}+a_{2}y_{1}=a_{1}+y_{2}a_{3}=x\in I$, and so $a_{1}\in I$,
that is, $a_{1}=0$, and $x=a_{2}y_{1}=y_{2}a_{3}$. Hence $\Delta(x)\in a_{2}\otimes y_{1}+y_{2}\otimes a_{3}+I\otimes I=1\otimes a_{2}y_{1}+y_{2}a_{3}\otimes1+I\otimes I$,
and so we have \[
\Delta(x)\in x\otimes1+1\otimes x+I\otimes I,\qquad\text{for all }x\in I.\]
The product $gh\in G$ is given by the element $(g\otimes h)\circ\Delta\in\Hom(A[\mathbf{G}],A)$.
Hence\[
gh(x)\in g(x)+h(x)+g(I)h(I)\subseteq g(x)+h(x)+\mathfrak{m}^{i+j},\qquad\text{for all }x\in I,\]
and so $(\phi_{i+j}\circ(gh-g-h))(I)=0$. Thus the map $\phi_{i+j}\circ(gh-g-h)$
factors through $\epsilon$, and since $\phi_{i+j}$ is the unique
$A$-algebra map $A\rightarrow A/\mathfrak{m}^{i+j}$, we must have
$\phi_{i+j}\circ(gh-g-h)=\phi_{i+j}\circ\epsilon$. This means exactly
that the element $gh-g-h\in\Hom(A[\mathbf{G}],A)$ lies in the kernel
$G^{i+j}$. We thus see that $gh=g+h=h+g=hg$, modulo $G^{i+j}$,
and the result follows. 

We now prove (b). Let $W$ be the group generated by simple reflections
corresponding to the root system of $\mathbf{G}$ relative to $\mathbf{T}$
(cf.~\cite{SGA3}, XXI 1.1.8). By \cite{SGA3} XXII 3.3 resp.~3.8
we have a natural inclusion $W\subseteq N_{\mathbf{G}}(\mathbf{T})(A)/\mathbf{T}(A)$
resp.~surjection $N_{\mathbf{G}}(\mathbf{T})(A)\rightarrow(N_{\mathbf{G}}(\mathbf{T})/\mathbf{T})(A)$.
For any $w\in W$ we can thus choose a lift $n_{w}\in N_{\mathbf{G}}(\mathbf{T})(A)$.
Since $A$ is local, we have \[
G=\bigcup_{w\in W}n_{w}U^{-}TU,\]
(cf.~\cite{SGA3}, XXII 5.7.4 (ii) and also 5.7.5 (ii)). In particular,
we may take $n_{1}=1$ as a representative for the trivial element
$1\in W$. 

Now, if $\phi_{1}(n_{w}u^{-}tu)=1$, for some $w\in W$, $u^{-}\in U^{-}$,
$t\in T$, $u\in U$, then $B_{1}\phi_{1}(n_{w})U_{1}^{-}\subseteq B_{1}U_{1}^{-}$,
and the Bruhat decomposition in $G_{1}$ with respect to the Borel
subgroups $B_{1}$ and $B_{1}^{-}$ implies that $\phi_{1}(n_{w})\in T_{1}$.
Hence $\phi_{1}(u^{-})=\phi_{1}(n_{w}t)=\phi_{1}(u)=1$. Let $k$
be the residue field of $A$. The morphism $ $$N_{\mathbf{G}}(\mathbf{T})\rightarrow N_{\mathbf{G}}(\mathbf{T})/\mathbf{T}$
yields a commutative diagram $$\begin{CD} N_{\mathbf{G}}(\mathbf{T})(A) @>>> (N_{\mathbf{G}}(\mathbf{T})/\mathbf{T})(A)\\
@V\phi_1VV @VV\phi_1V\\
N_{\mathbf{G}}(\mathbf{T})(k) @>>> (N_{\mathbf{G}}(\mathbf{T})/\mathbf{T})(k).\end{CD}$$ Since $\mathbf{G}$ is split reductive, its root datum is canonically
isomorphic to the root datum of $\mathbf{G}\times k$ (\cite{SGA3},
XXII 1.15 b)), and hence the map $\phi_{1}:(N_{\mathbf{G}}(\mathbf{T})/\mathbf{T})(A)\rightarrow(N_{\mathbf{G}}(\mathbf{T})/\mathbf{T})(k)$
restricts to an isomorphism between $W$ and the Weyl group of the
root datum of $\mathbf{G}\times k$ considered as a subgroup of $(N_{\mathbf{G}}(\mathbf{T})/\mathbf{T})(k)$.
The image of $\phi_{1}(n_{w})\in T_{1}\subseteq N_{\mathbf{G}}(\mathbf{T})(k)$
in $(N_{\mathbf{G}}(\mathbf{T})/\mathbf{T})(k)$ is trivial, and by
the commutativity of the above diagram, the image of $n_{w}$ in $W\subseteq N_{\mathbf{G}}(\mathbf{T})(A)/\mathbf{T}(A)$
is thus the trivial element. It follows from our choice of representatives
that $n_{w}=n_{1}=1$, whence $G^{1}=(U^{-})^{1}T^{1}U^{1}$. Uniqueness
follows immediately, since if $u^{-}tu=u_{1}^{-}t_{1}u_{1}$, then
$(u^{-})^{-1}u_{1}^{-}\in U^{-}\cap B=\{1\}$, and similarly, $t^{-1}t_{1}=u^{-1}u_{1}=1$.
\end{proof}
\begin{rem*}
For basic facts about Hopf algebras we have followed \cite{Jantzen},
2.4. In the present paper we apply Iwahori decomposition only in the
case of a reductive group scheme over a strictly Henselian base, and
such groups are split by Lemma~\ref{lem:splitting}.
\end{rem*}
We return to our situation where $\mathbf{G}$ is a reductive group
scheme over the Artinian local ring $A$, with finite residue field
$k$. Using the isomorphism $G=(\mathcal{F}_{A}\mathbf{G})(k)\cong\mathbf{G}(A)$
together with Lemma~\ref{lem:commutator-Iwahori} we get corresponding
commutator relations and Iwahori decomposition in $G$.

We now prove a result which is a form of Bruhat decomposition for
$G$, and is both a strengthened and a generalized form of a result
of Hill (cf.~\cite{Hill_regular}, 2.6). Let $\mathbf{U}^{-}$ (resp.
$\mathbf{U}'^{-}$) be the unipotent radical of a Borel subgroup of
$\mathbf{G}$ containing $\mathbf{T}$ (resp. $\mathbf{T}'$) such
that $\mathbf{U}\cap\mathbf{U}^{-}=\{1\}$ (resp. $\mathbf{U}'\cap\mathbf{U}'^{-}=\{1\}$).
Let $U^{-}=(\mathcal{F}_{A}\mathbf{U}^{-})(k)$ and $U'^{-}=(\mathcal{F}_{A}\mathbf{U}'^{-})(k)$
be the corresponding subgroups of $G$.

\begin{lem}
\label{lem:Unique decomposition}Let $U,U',U^{-},U'^{-}$ be as above.
Then $G$ decomposes as\[
G=\bigsqcup_{w\in W(T_{1},T_{1}')}(U\cap\hat{w}U'^{-}\hat{w}^{-1})\hat{w}T'((U'^{-})^{1}\cap\hat{w}^{-1}U^{-}\hat{w})U',\]
and every element $g\in G$ can be written uniquely in the form $g=u\hat{w}t'ku'$,
where $u\in U\cap\hat{w}U'^{-}\hat{w}^{-1}$, $t'\in T'$, $k\in(U'^{-})^{1}\cap\hat{w}^{-1}U^{-}\hat{w}$,
and $u_{}'\in U'$.
\end{lem}
\begin{proof}
In the case $r=1$ the result is a well-known consequence of Bruhat's
lemma. Using the surjection $\phi_{r}$ we lift the decomposition
to $G$, and so we may write \[
G=\bigsqcup_{w\in W(T_{1},T_{1}')}(U\cap\hat{w}U'^{-}\hat{w}^{-1})\hat{w}G^{1}T'U'.\]
Now, by Iwahori decomposition we have $G^{1}=(U'^{-})^{1}T^{1}(U')^{1}$,
so \[
(U\cap\hat{w}U'^{-}\hat{w}^{-1})\hat{w}G^{1}T'U'=(U\cap\hat{w}U'^{-}\hat{w}^{-1})\hat{w}(U'^{-})^{1}T'U'.\]
 The formula $U=\prod_{\alpha\in\Phi^{+}}U_{\alpha}$ implies that
we may write\[
(U'^{-})^{1}=((U'^{-})^{1}\cap\hat{w}^{-1}U\hat{w})((U'^{-})^{1}\cap\hat{w}^{-1}U^{-}\hat{w}),\]
and since $\hat{w}((U'^{-})^{1}\cap\hat{w}^{-1}U\hat{w})\hat{w}^{-1}\in(U\cap\hat{w}U'^{-}\hat{w}^{-1})$,
we have \[
G=\bigsqcup_{w\in W(T_{1},T_{1}')}(U\cap\hat{w}U'^{-}\hat{w}^{-1})\hat{w}((U'^{-})^{1}\cap\hat{w}^{-1}U^{-}\hat{w})T'U'.\]
Since $T'$ normalizes $(U'^{-})^{1}\cap\hat{w}^{-1}U^{-}\hat{w}$,
we get the desired decomposition. Now let $u\hat{w}t'ku'=u_{1}\hat{w}t'_{1}k_{1}u'_{1}$,
for $u,u_{1}\in U\cap\hat{w}U'^{-}\hat{w}^{-1}$, $t',t'_{1}\in T'$,
$k,k_{1}\in(U'^{-})^{1}\cap\hat{w}^{-1}U^{-}\hat{w}$, and $u',u'_{1}\in U'$.
Then $u'u_{1}'^{-1}=(u\hat{w}t'k)^{-1}u_{1}\hat{w}t_{1}'k_{1}$, and
since $u'u_{1}'^{-1}\in U'$ and $(u\hat{w}t'k)^{-1}u_{1}\hat{w}t_{1}'k_{1}\in T'U'^{-}$
we must have $u'=u'_{1}$ and $u\hat{w}t'k=u_{1}\hat{w}t_{1}'k_{1}$,
or equivalently $t'kk_{1}^{-1}t_{1}'^{-1}=\hat{w}^{-1}u^{-1}u_{1}\hat{w}$.
Since $t'kk_{1}^{-1}t_{1}'^{-1}\in T'\hat{w}^{-1}U{}^{-}\hat{w}=\hat{w}^{-1}TU^{-}\hat{w}$
and $\hat{w}^{-1}u^{-1}u_{1}\hat{w}\in\hat{w}^{-1}U\hat{w}$, we conclude
that $t'k=t'_{1}k_{1}$ and $u=u_{1}$. Thus also $t'=t'_{1}$ and
$k=k_{1}$, and the lemma is proved.
\end{proof}
If $\mathcal{T}$ is a commutative algebraic group over $\overline{\mathbb{F}}_{q}$
with fixed $\mathbb{F}_{q}$-structure and with Frobenius map $F:\mathcal{T}\rightarrow\mathcal{T}$,
then for any integer $n\geq1$ we have a norm map\[
N_{F}^{F^{n}}:\mathcal{T}^{F^{n}}\longrightarrow\mathcal{T}^{F},\qquad t\longmapsto tF(t)F^{2}(t)\cdots F^{n-1}(t).\]

\begin{lem}
\label{lem:Lusztig's 1.1}Let $\mathcal{T}$ and $\mathcal{T}'$ be
two commutative connected algebraic groups over $\overline{\mathbb{F}}_{q}$
with fixed $\mathbb{F}_{q}$-rational structures with Frobenius maps
$F:\mathcal{T}\rightarrow\mathcal{T},$ and $F:\mathcal{T}'\rightarrow\mathcal{T}'$.
Let $f:\mathcal{T}\iso\mathcal{T}'$ be an isomorphism of algebraic
groups over $\overline{\mathbb{F}}_{q}$. Let $n\geq1$ be such that
$F^{n}f=fF^{n}:\mathcal{T}\rightarrow\mathcal{T}'$; thus $f:\mathcal{T}^{F^{n}}\iso\mathcal{T}'^{F^{n}}$.
Let\[
H=\{(t,t')\in\mathcal{T}\times\mathcal{T}'\mid f(F(t)^{-1}t)=F(t')^{-1}t'\},\]
(a subgroup of $\mathcal{T}\times\mathcal{T}'$ containing $\mathcal{T}^{F}\times\mathcal{T}'^{F}$).
Let $\theta\in\widehat{\mathcal{T}^{F}}$ and $\theta'\in\widehat{\mathcal{T}^{'F}}$
be such that $\theta^{-1}\boxtimes\theta'$ is trivial on $(\mathcal{T}^{F}\times\mathcal{T}'^{F})\cap H^{0}$.
Then $\theta N_{F}^{F^{n}}=\theta'N_{F}^{F^{n}}f\in\widehat{\mathcal{T}^{F^{n}}}$
.
\end{lem}
\begin{proof}
See \cite{lusztig-preprint-published}, 1.1.
\end{proof}
From now on, let $\mathcal{T}=T^{r-1}$ and $\mathcal{T}'=T'^{r-1}$.
Note that in the case $r=1$ we have $\mathcal{T}=T$ and $\mathcal{T}'=T'$.

\begin{lem}
\label{lem:Lusztig's 1.3}Let $w\in W(T_{1},T_{1}')$, and let $\theta\in\widehat{T^{F}}$,
$\theta'\in\widehat{T'^{F}}$. Assume that $H_{c}^{j}(\Sigma_{w})_{\theta^{-1},\theta'}\neq0$
for some $j\in\mathbb{Z}$. Let $g=F(\hat{w})^{-1}$ and let $n\geq1$
be such that $g\in G^{F^{n}}$. Then $\mathrm{Ad}(g)$ (conjugation
by $g$) carries $\mathcal{T}^{F^{n}}$ onto $\mathcal{T}^{'F^{n}}$
and $\theta|_{\mathcal{T}^{F}}\circ N_{F}^{F^{n}}\in\widehat{\mathcal{T}^{F^{n}}}$
to $\theta'|_{\mathcal{T}'^{F}}\circ N_{F}^{F^{n}}\in\widehat{\mathcal{T}'^{F^{n}}}$
.
\end{lem}
\begin{proof}
Put $U_{\hat{w}}=U\cap\hat{w}U'^{-}\hat{w}^{-1}$ and $K=(U'^{-})^{1}\cap\hat{w}^{-1}U^{-}\hat{w}$.
$ $By Lemma \ref{lem:Unique decomposition}, we then have an isomorphism\[
\begin{split}\tilde{\Sigma}_{\hat{w}} & :=\{(x,x',u,u',k,\nu)\in F(U)\times F(U')\times U_{\hat{w}}\times U'\times K\times\hat{w}T'\mid\\
 & \hspace{.55cm}xF(u)F(\nu)F(k)F(u')=u\nu ku'x'\}\longiso\Sigma_{w},\end{split}
\]
given by $(x,x',u,u',k,\nu)\mapsto(x,x',u\nu ku')$. This isomorphism
is compatible with the $T^{F}\times T'^{F}$-actions, where $T^{F}\times T'^{F}$
acts on $\tilde{\Sigma}_{\hat{w}}$ by\renewcommand{\theequation}{\alph{equation}}\begin{equation}
(t,t'):(x,x',u,u',k,\nu)\longmapsto(txt^{-1},t'x't'^{-1},tut^{-1},t'u't'^{-1},t'kt'^{-1},t\nu t'^{-1}).\label{eq:(a)}\end{equation}
Hence we have $H_{c}^{j}(\tilde{\Sigma}_{\hat{w}})_{\theta^{-1},\theta'}\neq0$.
By the substitution $xF(u)\mapsto x$, $x'F(u')^{-1}\mapsto x'$,
the variety $\tilde{\Sigma}_{\hat{w}}$ is rewritten as\begin{equation}
\begin{split} & \{(x,x',u,u',k,\nu)\in F(U)\times F(U')\times U_{\hat{w}}\times U'\times K\times\hat{w}T'\mid\\
 & xF(k)F(\nu)=u\nu ku'x'\};\end{split}
\label{eq:(b)}\end{equation}
in these coordinates, the action of $T^{F}\times T'^{F}$ is still
given by \eqref{eq:(a)}. Let \[
H=\{(t,t')\in\mathcal{T}\times\mathcal{T}'\mid t'F(t')^{-1}=F(\hat{w})^{-1}tF(t)^{-1}F(\hat{w})\}\]
(a closed subgroup of $T\times T'$). It acts on the variety \eqref{eq:(b)}
by the same formula as in \eqref{eq:(a)} (we use Lemma \ref{lem:commutator-Iwahori}
to show that $\mathcal{T}$ and $\mathcal{T}'$ centralize $G^{1}$).
By \cite{delignelusztig}, 6.5, the induced action of $H$ on $H_{c}^{j}(\tilde{\Sigma}_{\hat{w}})$
is trivial when restricted to the connected component $H^{0}$. In
particular, the intersection $(T^{F}\times T'^{F})\cap H^{0}$ acts
trivially on $H_{c}^{j}(\tilde{\Sigma}_{\hat{w}})$. Since $H_{c}^{j}(\tilde{\Sigma}_{\hat{w}})_{\theta^{-1},\theta^{'}}\neq0$,
it follows that $\theta^{-1}\boxtimes\theta'$ is trivial on $(T^{F}\times T'^{F})\cap H^{0}$.
Let $g=F(\hat{w})^{-1}$ and let $n\geq1$ be such that $g\in G^{F^{n}}$.
Then $\mathrm{Ad}(g)$ carries $\mathcal{T}^{F^{n}}$ onto $\mathcal{T}'^{F^{n}}$
and (by Lemma \ref{lem:Lusztig's 1.1} with $f=\mathrm{Ad}(g)$) it
carries $\theta|_{\mathcal{T}^{F}}\circ N_{F}^{F^{n}}$ to $ $$\theta'|_{\mathcal{T}'^{F}}\circ N_{F}^{F^{n}}$. 
\end{proof}
With the above lemma proved for each $\Sigma_{w}$ we can deduce a
similar statement for the whole variety $\Sigma$. This will be used
later (in Prop.~\ref{pro:Geom conj}) to prove a result which is
a generalization of Theorem~6.2 of Deligne and Lusztig \cite{delignelusztig}.

\begin{lem}
\label{lem:Lusztig's 1.4}Let $\theta\in\widehat{T^{F}}$, $\theta'\in\widehat{T'^{F}}$
be such that\setcounter{equation}{0}\renewcommand{\theequation}{\alph{equation}}
\begin{equation}
H_{c}^{j}(\Sigma)_{\theta^{-1},\theta'}\neq0\label{eq:Lusztig's 1.4}\end{equation}
for some $j\in\mathbb{Z}$. Then there exists $n\geq1$ and $g\in N(T',T)^{F^{n}}$
such that $\mathrm{Ad}(g)$ carries $\theta|_{\mathcal{T}^{F}}\circ N_{F}^{F^{n}}\in\widehat{\mathcal{T}^{F^{n}}}$
to $\theta'|_{\mathcal{T}'^{F}}\circ N_{F}^{F^{n}}\in\widehat{\mathcal{T}'^{F^{n}}}$
.
\end{lem}
\begin{proof}
It is well-known that the subvarieties $G_{1,w}$ of $G_{1}$ have
the following property: for some ordering $\leq$ of $W(T_{1},T_{1}'),$
the unions $\bigcup_{w'\leq w}G_{1,w'}$ are closed in $G_{1}$. It
follows that the unions $\bigcup_{w'\leq w}G_{w'}$ are closed in
$G$, and the unions $\bigcup_{w'\leq w}\Sigma_{w'}$ are closed in
$\Sigma$. The spectral sequence associated to the filtration of $\Sigma$
by these unions, together with \eqref{eq:Lusztig's 1.4}, shows that
there exists $w\in W(T_{1},T_{1}')$ and $j'\in\mathbb{Z}$ such that
$H_{c}^{j'}(\Sigma_{w})_{\theta^{-1},\theta'}\neq0$. We can therefore
apply Lemma \ref{lem:Lusztig's 1.3} to get an element $g=F(\hat{w})^{-1}\in N(T',T)^{F^{n}}$,
for some $n\geq1$, satisfying the conclusion.
\end{proof}
For each root $\alpha\in\Phi(\mathbf{G},\mathbf{T})$ we have a unique
coroot $\check{\alpha}\in\Hom_{A\text{-gr}}((\mathbb{G}_{m})_{A},\mathbf{T})$.
Let $\mathbf{T}^{\alpha}$ denote the image of $\check{\alpha}$ in
$\mathbf{T}$, so that $\mathbf{T}^{\alpha}$ is a $1$-dimensional
torus in $\mathbf{T}$ (cf.~\cite{SGA3}, XX 3). Keeping with our
notational conventions, we let $U_{\alpha}=(\mathcal{F}_{A}\mathbf{U}_{\alpha})(k)$
and $T^{\alpha}=(\mathcal{F}_{A}\mathbf{T}^{\alpha})(k)$. We also
write $\mathcal{T}^{\alpha}=(T^{\alpha})^{r-1}$ (a $1$-dimensional
subgroup of $\mathcal{T}=T^{r-1}$, cf.~\cite{greenberg2}, 3).

\begin{defn}
\label{def:regular}Let $\chi\in\widehat{\mathcal{T}^{F}}$. We say
that $\chi$ is \emph{regular} if for any $\alpha\in\Phi$ and any
$n\geq1$ such that $F^{n}(\mathcal{T}^{\alpha})=\mathcal{T}^{\alpha}$,
the restriction of $\chi\circ N_{F}^{F^{n}}:\mathcal{T}^{F^{n}}\rightarrow\overline{\mathbb{Q}}_{l}\vphantom{\overline{\mathbb{Q}}}^{\hspace{-3pt}\times}$
to $(\mathcal{T}^{\alpha})^{F^{n}}$ is non-trivial. If $\theta\in\widehat{T^{F}}$,
we say that $\theta$ is \emph{regular} if $\theta|_{\mathcal{T}^{F}}$
is regular.
\end{defn}
\begin{lem}
\label{lem:regularity test}Let $\chi\in\widehat{\mathcal{T}^{F}}$,
and suppose that there exists an $n\geq1$ such that for all $\alpha\in\Phi$,
$F^{n}(\mathcal{T}^{\alpha})=\mathcal{T}^{\alpha}$ and the restriction
of $\chi\circ N_{F}^{F^{n}}$ to $(\mathcal{T}^{\alpha})^{F^{n}}$
is non-trivial. Then $\chi$ is regular.
\end{lem}
\begin{proof}
We first show some properties of the norm map. Let $\mathcal{T}$
be a commutative algebraic group defined over $\mathbb{F}_{q}$ with
Frobenius $F$. Let $a$ and $b$ be two positive integers such that
$b=ka$, for some integer $k$. Then clearly $\mathcal{T}^{F^{a}}\subseteq\mathcal{T}^{F^{b}}$.
We extend the definition of the norm map by defining the map $N_{F^{a}}^{F^{b}}:\mathcal{T}^{F^{b}}\rightarrow\mathcal{T}^{F^{a}}$,
$x\mapsto xF^{a}(x)F^{2a}(x)\cdots F^{(k-1)a}(x)$. We then have \[
N_{F}^{F^{a}}N_{F^{a}}^{F^{b}}(x)=\prod_{j=0}^{a-1}F^{j}\left(\prod_{i=0}^{k-1}F^{ia}(x)\right)=\prod_{j=0}^{a-1}\prod_{i=0}^{k-1}F^{j+ia}(x)=\prod_{i=0}^{b-1}F^{i}(x)=N_{F}^{F^{b}}(x),\]
so $N_{F}^{F^{b}}=N_{F}^{F^{a}}\circ N_{F^{a}}^{F^{b}}$. Now suppose
that $\mathcal{H}$ is a closed connected subgroup of $\mathcal{T}$
which is stable under $F^{a}$ and $F^{b}$. The map $N_{F^{a}}^{F^{b}}$
restricts to a map $N_{F^{a}}^{F^{b}}:\mathcal{H}^{F^{b}}\rightarrow\mathcal{H}^{F^{a}}$,
which we claim is surjective. Indeed, if $x\in\mathcal{H}^{F^{a}}$,
then by the Lang-Steinberg theorem there exists some $y\in\mathcal{H}$
such that $y^{-1}F^{b}(y)=x$. Now, $F^{a}(x)=x$ implies that $F^{a}(y^{-1}F^{b}(y))=y^{-1}F^{b}(y)$,
and so $F^{b}(y^{-1}F^{a}(y))=y^{-1}F^{a}(y)$. Thus $y^{-1}F^{a}(y)\in\mathcal{H}^{F^{b}}$,
and $N_{F^{a}}^{F^{b}}(y^{-1}F^{a}(y))=y^{-1}F^{b}(y)=x$. 

Now let $m$ be the minimal positive integer such that $F^{m}(\mathcal{T}^{\alpha})=\mathcal{T}^{\alpha}$,
for all $\alpha\in\Phi$. Write $n=gm+h$ with integers $g\geq1$
and $0\leq h<m$. Then $F^{n}(\mathcal{T}^{\alpha})=\mathcal{T}^{\alpha}\,\forall\alpha$
implies that $F^{h}(\mathcal{T}^{\alpha})=\mathcal{T}^{\alpha}\,\forall\alpha$
, so the minimality of $m$ forces $h=0$. If for some $\alpha$ we
have $\chi\circ N_{F}^{F^{m}}((\mathcal{T}^{\alpha})^{F^{m}})=1$,
then $N_{F}^{F^{n}}=N_{F}^{F^{m}}\circ N_{F^{m}}^{F^{n}}$ implies
$N_{F}^{F^{n}}((\mathcal{T}^{\alpha})^{F^{n}})\subseteq N_{F}^{F^{m}}((\mathcal{T}^{\alpha})^{F^{m}})$,
so $\chi\circ N_{F}^{F^{n}}((\mathcal{T}^{\alpha})^{F^{n}})=1$, which
contradicts the hypothesis. Thus $m$ is such that the restriction
of $\chi\circ N_{F}^{F^{m}}$ to $(\mathcal{T}^{\alpha})^{F^{m}}$
is non-trivial, for all $\alpha$.

Finally, suppose that $m'$ is an arbitrary positive integer such
that $ $$F^{m'}(\mathcal{T}^{\alpha})=\mathcal{T}^{\alpha}$, for
all $\alpha$. Then as we have seen, $m\mid m'$. By the surjectivity
and transitivity of the norm map we get $N_{F}^{F^{m'}}((\mathcal{T}^{\alpha})^{F^{m'}})=N_{F}^{F^{m}}\circ N_{F^{m}}^{F^{m'}}((\mathcal{T}^{\alpha})^{F^{m'}})=N_{F}^{F^{m}}((\mathcal{T}^{\alpha})^{F^{m}})$.
Thus $\chi\circ N_{F}^{F^{m'}}((\mathcal{T}^{\alpha})^{F^{m'}})=\chi\circ N_{F}^{F^{m}}((\mathcal{T}^{\alpha})^{F^{m}})\neq1$
for all $\alpha$, and so $\chi$ is regular. 
\end{proof}
As before, $\mathbf{U}$ is the unipotent radical of a Borel subgroup
of $\mathbf{G}$ containing $\mathbf{T}$. Let $\mathbf{V}$ be the
unipotent radical of another such Borel subgroup, and let $ $$\mathbf{U}^{-}$
(resp. $\mathbf{V}^{-}$) be the unipotent radical of a Borel subgroup
of $\mathbf{G}$ containing $\mathbf{T}$ such that $\mathbf{U}\cap\mathbf{U}^{-}=\{1\}$
(resp. $\mathbf{V}\cap\mathbf{V}^{-}=\{1\}$). The corresponding subgroups
of $G$ are denoted by $U,U^{-},V,V^{-}$, respectively. Let \[
\Phi^{+}=\{\alpha\in\Phi\mid\mathbf{U}_{\alpha}\subseteq\mathbf{V}\},\qquad\Phi^{-}=\{\alpha\in\Phi\mid\mathbf{U}_{\alpha}\subseteq\mathbf{V}^{-}\},\]
be the positive, respectively negative, roots corresponding to the
choice of $\mathbf{V}$ and $\mathbf{V}^{-}$. Then $\Phi=\Phi^{+}\sqcup\Phi^{-}$
and $\Phi^{-}=\{-\alpha\mid\alpha\in\Phi^{+}\}$. For $\alpha\in\Phi^{+}$
let $\height(\alpha)$ be the largest integer $n\geq1$ such that
$\alpha=\alpha_{1}+\alpha_{2}+\dots+\alpha_{n}$ with $\alpha_{i}\in\Phi^{+}$.
In the following, for two elements $x,y$ of a group, we shall write
$[x,y]=xyx^{-1}y^{-1}$ for their commutator. 

\begin{lem}
\label{lem:roots-commutator}Let $x\in(U_{\alpha})^{b}$, and $x'\in(U_{\alpha'})^{c}$,
where $\alpha,\alpha'\in\Phi$ and $0\leq b,c\leq r$. Then the following
holds:
\begin{enumerate}
\item If $b+c\geq r$ then $xx'=x'x$.
\item If $b+c\leq r$ and $\alpha\neq-\alpha'$, then \[
[x,x']=\prod_{\substack{i,i'\geq1\\
i\alpha+i'\alpha'\in\Phi}
}u_{i,i'},\]
where $u_{i,i'}\in(U_{i\alpha+i'\alpha'})^{b+c}$. (The factors in
the product are written in a fixed but arbitrary order.)
\item If $b+c\geq r-1$, $b+2c\geq r$, and $\alpha=-\alpha'$, then $[x,x']=\tau_{x,x'}u$,
where $\tau_{x,x'}\in\mathcal{T}^{\alpha}$ and $u\in(U_{\alpha})^{r-1}$
are uniquely determined. 
\end{enumerate}
\end{lem}
\begin{proof}
Part (a) follows immediately from Lemma \ref{lem:commutator-Iwahori}(a).
Part (b) is Chevalley's commutator formula (cf.~\cite{Demazure},
3.3.4) applied to the various subgroups $\mathbf{U}_{i\alpha+i'\alpha'}(A)$
of $\mathbf{G}(A)$. For each $\alpha,$ choose corresponding isomorphisms
$p_{\alpha}:(\mathbb{G}_{a})_{A}\rightarrow\mathbf{U}_{\alpha}$ as
in \cite{SGA3}, XX 1.20. Functorial properties then imply that $p_{\alpha}(\Ker(\mathbb{G}_{a}(A)\rightarrow\mathbb{G}_{a}(A/\mathfrak{m}^{j})))=\Ker(\mathbf{U}_{\alpha}(A)\rightarrow\mathbf{U}_{\alpha}(A/\mathfrak{m}^{j})\cong(U_{\alpha})^{j}$,
for any $0\leq j\leq r$ (note that we abuse notation since $p_{\alpha}$
is really a map of group functors rather than groups), and the formula
follows. Finally, we prove (c). Let $\tilde{x},\tilde{x}'\in\mathbb{G}_{a}(A)$
be such that $p_{\alpha}(\tilde{x})=x$ and $p_{-\alpha}(\tilde{x'})=x'$.
Then $\tilde{x}\in\Ker(\mathbb{G}_{a}(A)\rightarrow\mathbb{G}_{a}(A/\mathfrak{m}^{b}))$,
and $\tilde{x}'\in\Ker(\mathbb{G}_{a}(A)\rightarrow\mathbb{G}_{a}(A/\mathfrak{m}^{c}))$,
respectively; thus $1+a\tilde{x}\tilde{x}'\in\mathbb{G}_{m}(A)$,
for any $a\in\mathbb{G}_{m}(A)$. The hypotheses $b+c\geq r-1$ and
$b+2c\geq r$ imply that $\tilde{x}\tilde{x}'\in\Ker(\mathbb{G}_{a}(A)\rightarrow\mathbb{G}_{a}(A/\mathfrak{m}^{r-1}))$
and $\tilde{x}\tilde{x}'^{2}=0$. By \cite{SGA3}, XX 2.2 we have,
for some $a\in\mathbb{G}_{m}(A)$:\[
p_{\alpha}(\tilde{x})p_{-\alpha}(\tilde{x}')=p_{-\alpha}(\frac{\tilde{x}'}{1+a\tilde{x}\tilde{x}'})\check{\alpha}(1+a\tilde{x}\tilde{x}')p_{\alpha}(\frac{\tilde{x}}{1+a\tilde{x}\tilde{x}'}).\]
From this formula we get

\begin{align*}
[x,x'] & =p_{\alpha}(\tilde{x})p_{-\alpha}(\tilde{x}')p_{\alpha}(-\tilde{x})p_{-\alpha}(-\tilde{x}')\\
 & =p_{-\alpha}(\frac{\tilde{x}'}{1+a\tilde{x}\tilde{x}'})\check{\alpha}(1+a\tilde{x}\tilde{x}')p_{\alpha}(\frac{\tilde{x}}{1+a\tilde{x}\tilde{x}'})\\
 & \hspace{.45cm}p_{-\alpha}(\frac{-\tilde{x}'}{1+a\tilde{x}\tilde{x}'})\check{\alpha}(1+a\tilde{x}\tilde{x}')p_{\alpha}(\frac{-\tilde{x}}{1+a\tilde{x}\tilde{x}'})\\
 & =\check{\alpha}(1+a\tilde{x}\tilde{x}')^{2}p_{-\alpha}((1+a\tilde{x}\tilde{x}')^{4}\frac{\tilde{x}'}{1+a\tilde{x}\tilde{x}'})p_{\alpha}((1+a\tilde{x}\tilde{x}')^{2}\frac{\tilde{x}}{1+a\tilde{x}\tilde{x}'})\\
 & \hspace{.45cm}p_{-\alpha}((1+a\tilde{x}\tilde{x}')^{2}\frac{-\tilde{x}'}{1+a\tilde{x}\tilde{x}'})p_{\alpha}(\frac{-\tilde{x}}{1+a\tilde{x}\tilde{x}'})\\
 & \hspace{.45cm}\text{(using repeatedly the facts that }\forall\, t\in\mathbf{T}(A),z\in\mathbb{G}_{a}(A),\alpha\in\Phi,\text{ we have}\\
 & \hspace{.45cm}tp_{\alpha}(z)t^{-1}=p_{\alpha}(\alpha(t)z),\text{ and }\alpha\check{\alpha}(z)=z^{2}\text{)}\\
 & =\check{\alpha}(1+a\tilde{x}\tilde{x}')^{2}p_{-\alpha}(\tilde{x}'(1+a\tilde{x}\tilde{x}')^{3})p_{\alpha}(\tilde{x}(1+a\tilde{x}\tilde{x}'))\\
 & \hspace{.45cm}p_{-\alpha}(-\tilde{x}'(1+a\tilde{x}\tilde{x}'))p_{\alpha}(\frac{-\tilde{x}}{1+a\tilde{x}\tilde{x}'})\\
 & =\check{\alpha}(1+a\tilde{x}\tilde{x}')^{2}p_{-\alpha}(\tilde{x}')p_{\alpha}(\tilde{x}(1+a\tilde{x}\tilde{x}'))p_{-\alpha}(-\tilde{x}')p_{\alpha}(\frac{-\tilde{x}}{1+a\tilde{x}\tilde{x}'})\\
 & \hspace{.45cm}\text{(using }\tilde{x}\tilde{x}'^{2}=0\text{)}\\
 & =\check{\alpha}(1+a\tilde{x}\tilde{x}')^{2}p_{-\alpha}(\tilde{x}')p_{-\alpha}(\frac{-\tilde{x}'}{1-a\tilde{x}'\tilde{x}(1+\tilde{x}\tilde{x}')})\\
 & \hspace{.45cm}\check{\alpha}(1-a\tilde{x}'\tilde{x}(1+\tilde{x}\tilde{x}'))p_{\alpha}(\frac{\tilde{x}(1+\tilde{x}\tilde{x}')}{1-a\tilde{x}'\tilde{x}(1+\tilde{x}\tilde{x}')})p_{\alpha}(\frac{-\tilde{x}}{1+a\tilde{x}\tilde{x}'})\\
 & =\check{\alpha}(1+a\tilde{x}\tilde{x}')^{2}p_{-\alpha}(\tilde{x}')p_{-\alpha}(\frac{-\tilde{x}'}{1-a\tilde{x}\tilde{x}'})\\
 & \hspace{.45cm}\check{\alpha}(1-a\tilde{x}\tilde{x}')p_{\alpha}(\frac{\tilde{x}(1+a\tilde{x}\tilde{x}')}{1-a\tilde{x}\tilde{x}'})p_{\alpha}(\frac{-\tilde{x}}{1+a\tilde{x}\tilde{x}'})\\
 & =\check{\alpha}(1+a\tilde{x}\tilde{x}')^{2}\check{\alpha}(1-a\tilde{x}\tilde{x}')p_{-\alpha}(\tilde{x}'(1-a\tilde{x}\tilde{x}')^{2})p_{-\alpha}(\frac{-\tilde{x}'}{1-a\tilde{x}\tilde{x}'}(1-a\tilde{x}\tilde{x}')^{2})\\
 & \hspace{.45cm}p_{\alpha}(\frac{\tilde{x}(1+a\tilde{x}\tilde{x}')}{1-a\tilde{x}\tilde{x}'})p_{\alpha}(\frac{-\tilde{x}}{1+a\tilde{x}\tilde{x}'})\\
 & =\check{\alpha}((1+a\tilde{x}\tilde{x}')^{2}(1-a\tilde{x}\tilde{x}'))p_{-\alpha}(0)p_{\alpha}(\frac{\tilde{x}\tilde{x}'}{1-a\tilde{x}\tilde{x}'})\\
 & =\check{\alpha}(1+a\tilde{x}\tilde{x}')p_{\alpha}(\frac{\tilde{x}\tilde{x}'}{1-a\tilde{x}\tilde{x}'}).\end{align*}
Now \[
\check{\alpha}(1+a\tilde{x}\tilde{x}')\in\Ker(\mathbf{T}^{\alpha}(A)\rightarrow\mathbf{T}^{\alpha}(A/\mathfrak{m}^{r-1}))\cong\mathcal{T}^{\alpha}\]
and \[
p_{\alpha}(\frac{\tilde{x}\tilde{x}'}{1-a\tilde{x}\tilde{x}'})\in\Ker(\mathbf{U}_{\alpha}(A)\rightarrow\mathbf{U}_{\alpha}(A/\mathfrak{m}^{r-1}))\cong(U_{\alpha})^{r-1}.\]
Using the canonical isomorphism $G\cong\mathbf{G}(A)$, we conclude
that for elements $x\in(U_{\alpha})^{b}$, and $x'\in(U_{-\alpha})^{c}$
we have $[x,x']\in\mathcal{T}^{\alpha}(U_{\alpha})^{r-1}$. Finally,
because of the semidirect product $TU$ in $G$, the decomposition
of $[x,x']$ as an element of $\mathcal{T}^{\alpha}(U_{\alpha})^{r-1}$
is unique. 
\end{proof}
\begin{lem}
\label{lem: induction-commutator}We fix an order on $\Phi^{+}$.
For any $z\in V$, and $\beta\in\Phi^{+}$, define elements $x_{\beta}^{z}\in U_{\beta}$
via the decomposition $z=\prod_{\beta\in\Phi^{+}}x_{\beta}^{z}$ (factors
written using the given order on $\Phi^{+}$). Let $\alpha\in\Phi^{-}$
and $a$ be an integer such that $1\leq a\leq r-1$. Suppose that
$z\in V^{a}$ is an element such that $x_{\beta}^{z}\in(U_{\beta})^{a+1}$,
for all $\beta\in\Phi^{+}$ with $\height(\beta)>\height(-\alpha)$.
Then for any $\xi\in(U_{\alpha})^{r-a-1}$, we have\[
[\xi,z]=\tau_{\xi,z}\omega_{\xi,z},\quad\text{where }\tau_{\xi,z}\in\mathcal{T}^{\alpha}\text{ and }\omega_{\xi,z}\in(V^{-})^{r-1}.\]

\end{lem}
\begin{proof}
We argue by induction on $N_{z}=\#\{\beta\in\Phi^{+}\mid x_{\beta}^{z}\neq1\}$.
If $N_{z}=0$ the result is clear. Assume now that $N_{z}=1$ so that
$z\in U_{\beta}$ with $\beta\in\Phi^{+}$. If $\alpha=-\beta$ the
result follows from Lemma \ref{lem:roots-commutator}(c). If $\alpha\neq-\beta$
and $\height(\beta)>\height(-\alpha)$, then $z\in(U_{\alpha})^{a+1}$
and $\xi z=z\xi$ by Lemma \ref{lem:roots-commutator}(b). If $\alpha\neq-\beta$
and $\height(\beta)\leq\height(-\alpha)$, then by Lemma \ref{lem:roots-commutator}(b)
we have $[\xi,z]=\prod_{i,i'\geq1,i\alpha+i'\beta\in\Phi}u_{i,i'}$
with $u_{i,i'}\in(U_{i\alpha+i'\beta})^{r-1}$, and it is enough to
show that if $i,i'\geq1$, we cannot have $i\alpha+i'\beta\in\Phi^{+}$.
Now if we had $ $$i\alpha+i'\beta\in\Phi^{+}$ for some $i,i'\geq1$,
then general properties of root systems imply that $\alpha+\beta\in\Phi^{+}$,
and hence we would have $ $ $\height(\beta)>\height(-\alpha)$; contradiction.

Assume now that $N_{z}\geq2$, and that the assertion is true for
all $z'$ such that $N_{z'}\leq N_{z}$. We can write $z=z'z''$ where
$z',z''\in V^{a}$, $N_{z'}<N_{z}$, $N_{z''}<N_{z}$. Using the induction
hypothesis, we have \[
\xi z=\xi z'z''=\tau_{\xi,z'}\omega_{\xi,z'}z'\xi z''=\tau_{\xi,z'}\omega_{\xi,z'}z'\tau_{\xi,z''}\omega_{\xi,z''}z''\xi,\]
where $\tau_{\xi,z'}\in\mathcal{T}^{\alpha}$ and $\omega_{\xi,z'}\in(V^{-})^{r-1}$.
Since $a+r-1\geq r$, we have $z'\tau_{\xi,z''}=\tau_{\xi,z''}z'$
and $z'\omega_{\xi,z''}=\omega_{\xi,z''}z'$, by Lemma \ref{lem:commutator-Iwahori}.
Hence\[
\xi z=\tau_{\xi,z'}\omega_{\xi,z'}z'\tau_{\xi,z''}\omega_{\xi,z''}z''\xi=\tau_{\xi,z'}\omega_{\xi,z'}\tau_{\xi,z''}\omega_{\xi,z''}z'z''\xi=\tau_{\xi,z'}\tau_{\xi,z''}\omega_{\xi,z'}\omega_{\xi,z''}z\xi,\]
and so \[
[\xi,z]=\tau_{\xi,z}\omega_{\xi,z},\]
where $\tau_{\xi,z}=\tau_{\xi,z'}\tau_{\xi,z''}$ and $\omega_{\xi,z}=\omega_{\xi,z'}\omega_{\xi,z''}$. 
\end{proof}
Let $\mathbf{Z}=\mathbf{U}^{-}\cap\mathbf{V}$, $Z=(\mathcal{F}_{A}\mathbf{Z})(k)=U^{-}\cap V$,
and $\Phi'=\{\beta\in\Phi\mid\mathbf{U}_{\beta}\subseteq\mathbf{Z}\}$.
We obviously have $\Phi'\subseteq\Phi^{+}$. Let $\mathcal{X}$ be
the set of all subsets $I\subseteq\Phi'$ such that $I\neq\emptyset$
and $\height:\Phi^{+}\rightarrow\mathbb{N}$ is constant on $I$.

To any $z\in Z^{1}-\{1\}$ we associate a pair $(a,I_{z})$ where
$a$ is an integer $1\leq a\leq r-1$, and $I_{z}\in\mathcal{X}$,
as follows. We define $a$ by the condition that $z\in Z^{a,*}$.
If $x_{\beta}^{z}\in U_{\beta}$ are defined as in Lemma \ref{lem: induction-commutator}
in terms of a fixed order on $\Phi^{+}$, then $x_{\beta}^{z}\in(U_{\beta})^{a}$
for all $\beta\in\Phi'$ and $x_{\beta}^{z}=1$ for all $\beta\in\Phi^{+}-\Phi'$
(this is a consequence of the formula $Z^{a}=\prod_{\beta\in\Phi'}U_{\beta}^{a}$).
We define the set $I_{z}$ as \[
I_{z}=\{\alpha'\in\Phi'\mid x_{\alpha'}^{z}\in(U_{\alpha'})^{a,*}\text{ and }x_{\beta}^{z}\in(U_{\beta})^{a+1}\ \forall\,\beta\in\Phi^{+}\ \text{s.\ t.}\,\height(\beta)>\height(\alpha')\}.\]
The definition of $I_{z}$ does not depend on the choice of order
on $\Phi^{+}$. For any integer $1\leq a\leq r-1$ and $I\in\mathcal{X}$,
let $Z^{a,*,I}$ be the set of all $z\in Z^{1}-\{1\}$ such that $z\in Z^{a,*}$
and $I=I_{z}$. Thus we have a partition \setcounter{equation}{0}\renewcommand{\theequation}{\fnsymbol{equation}}
\begin{equation}
Z^{1}-\{1\}=\bigsqcup_{\substack{1\leq a\leq r-1\\
I\in\mathcal{X}}
}Z^{a,*,I}.\label{eq:(*)}\end{equation}

\section{The main results}

Recall the definitions of the groups $T,T',U,U',\mathcal{T},\mathcal{T}'$
and the variety $\Sigma$, from Section \ref{sec:Lemmas}. After having
set up the general framework, we are now ready to give results generalizing
those in \cite{lusztig-preprint-published}, with the structures of
the proofs remaining more or less the same. All the ideas of the proofs
in this section are due to Lusztig. The only thing that requires a
comment here is the use of the elements $\hat{w}$. In \cite{lusztig-preprint-published},
the inclusion $G_{1}\subseteq G$ (in our notation) allows one to
view the elements of $N(T_{1},T_{1}')$ as elements of $N(T,T')$.
However, in the general case which we consider here there is no such
inclusion, and instead we have to use lifts $\hat{w}\in N(T,T')$
of the elements $\dot{w}\in N(T_{1},T_{1}')$. The following Theorem~\ref{thm:Big-thm}
does not depend on the choice of lift $\hat{w}$ for each $\dot{w}$.
This can be seen from the proof, because we show that \[
\sum_{j\in\mathbb{Z}}(-1)^{j}\dim H_{c}^{j}(\Sigma_{w})_{\theta^{-1},\theta'}=\sum_{j\in\mathbb{Z}}(-1)^{j}\dim H_{c}^{j}(\hat{\Sigma}_{\hat{w}})_{\theta^{-1},\theta'},\]
where $\hat{\Sigma}_{\hat{w}}$ is the variety defined below, $\hat{w}$
is an arbitrary lift of $\dot{w}$, and the latter sum is equal to
$1$ if $F(w)=w$ and $\text{Ad}(\hat{w}):T'^{F}\rightarrow T^{F}$
carries $\theta$ to $\theta'$, and equals $0$ otherwise. Thus,
if $\hat{w}'$ is another lift of $\dot{w}$, then $\sum_{j\in\mathbb{Z}}(-1)^{j}\dim H_{c}^{j}(\hat{\Sigma}_{\hat{w}})_{\theta^{-1},\theta'}=\sum_{j\in\mathbb{Z}}(-1)^{j}\dim H_{c}^{j}(\hat{\Sigma}_{\hat{w}'})_{\theta^{-1},\theta'}$,
and so whenever $F(w)=w$, we see that $\text{Ad}(\hat{w}):T'^{F}\rightarrow T^{F}$
carries $\theta$ to $\theta'$ if and only if $\text{Ad}(\hat{w}'):T'^{F}\rightarrow T^{F}$
carries $\theta$ to $\theta'$.

\begin{thm}
\label{thm:Big-thm}Let $\theta\in\widehat{T^{F}}$ and $ $$\theta'\in\widehat{T'^{F}}$.
If $r\geq2$, assume that $\theta'$ is regular. Then $\sum_{j\in\mathbb{Z}}(-1)^{j}\dim H_{c}^{j}(\Sigma)_{\theta^{-1},\theta'}$
is equal to the number of $w\in W(T_{1},T_{1}')^{F}$ such that $\mathrm{Ad}(\hat{w}):T'^{F}\rightarrow T^{F}$
carries $\theta$ to $\theta'$.
\end{thm}
\begin{proof}
Using the partition $\Sigma=\bigsqcup_{w}\Sigma_{w}$ and the additivity
of Lefschetz numbers (cf.~\cite{dignemichel} 10.7) we see that it
is enough to prove that $ $$\sum_{j\in\mathbb{Z}}(-1)^{j}\dim H_{c}^{j}(\Sigma_{w})_{\theta^{-1},\theta'}$
is equal to $1$ if $F(w)=w$ and $\text{Ad}(\hat{w}):T'^{F}\rightarrow T^{F}$
carries $\theta$ to $\theta'$, and equals $0$ otherwise. We now
fix $w\in W(T_{1},T_{1}')$. We have\[
\Sigma_{w}=\{(x,x',y)\in F(U)\times F(U')\times G\mid xF(y)=yx',\ y\in UG^{1}\hat{w}T'U'=UZ^{1}\hat{w}T'U'\},\]
where $Z^{1}=(U^{-})^{1}\cap\hat{w}(U'{}^{-})^{1}\hat{w}^{-1}$ (the
equality $ $$UG^{1}\hat{w}T'U'=UZ^{1}\hat{w}T'U'$ follows from Lemma~\ref{lem:Unique decomposition}).
Let \begin{align*}
\hat{\Sigma}_{\hat{w}} & =\{(x,x',u,u',z,\tau')\in F(U)\times F(U')\times U\times U'\times Z^{1}\times T'\mid\\
 & \hspace{.45cm}xF(u)F(z)F(\hat{w})F(\tau')F(u')=uz\hat{w}\tau'u'x'\}.\end{align*}
The map $\hat{\Sigma}_{\hat{w}}\rightarrow\Sigma_{w}$ given by $(x,x',u,u',z,\tau')\mapsto(x,x',uz\hat{w}\tau'u')$
is a locally trivial fibration with all fibres isomorphic to a fixed
affine space. This map is compatible with the $T^{F}\times T'^{F}$-actions
where $T^{F}\times T'^{F}$ acts on $\hat{\Sigma}_{\hat{w}}$ by\renewcommand{\theequation}{\alph{equation}}\begin{align}
(t,t') & :(x,x',u,u',z,\tau')\label{eq:Big-thm (a)}\\
 & \longmapsto(txt^{-1},t'x't'^{-1},tut^{-1},t'u't'^{-1},tzt^{-1},\hat{w}^{-1}t\hat{w}\tau't'^{-1}).\nonumber \end{align}
Hence, by \cite{lusztigexpo}, 1.9 it is enough to show that $\sum_{j\in\mathbb{Z}}(-1)^{j}\dim H_{c}^{j}(\hat{\Sigma}_{\hat{w}})_{\theta^{-1},\theta'}$
is equal to $1$ if $F(w)=w$ and $\text{Ad}(\hat{w}):T'^{F}\rightarrow T^{F}$
carries $\theta$ to $\theta'$, and equals $0$ otherwise.

By the change of variables $xF(u)\mapsto x$, $x'F(u')^{-1}\mapsto x'$
we may rewrite $\hat{\Sigma}_{\hat{w}}$ as \begin{align*}
\hat{\Sigma}_{\hat{w}} & =\{(x,x',u,u',z,\tau')\in F(U)\times F(U')\times U\times U'\times Z^{1}\times T'\mid\\
 & \hspace{.45cm}xF(z)F(\hat{w})F(\tau')=uz\hat{w}\tau'u'x'\},\end{align*}
with the $T^{F}\times T'^{F}$-action still given by \eqref{eq:Big-thm (a)}.
We have a partition $\hat{\Sigma}_{\hat{w}}=\hat{\Sigma}'_{\hat{w}}\sqcup\hat{\Sigma}''_{\hat{w}}$,
where\begin{align*}
\hat{\Sigma}_{\hat{w}}' & =\{(x,x',u,u',z,\tau')\in F(U)\times F(U')\times U\times U'\times(Z^{1}-\{1\})\times T'\mid\\
 & \hspace{.45cm}xF(z)F(\hat{w})F(\tau')=uz\hat{w}\tau'u'x'\},\\
\hat{\Sigma}_{\hat{w}}'' & =\{(x,x',u,u',1,\tau')\in F(U)\times F(U')\times U\times U'\times\{1\}\times T'\mid\\
 & \hspace{.45cm}xF(\hat{w})F(\tau')=u\hat{w}\tau'u'x'\},\end{align*}
are stable under the $T^{F}\times T'^{F}$-action. It is then enough
to show that\begin{align}
 & \sum_{j\in\mathbb{Z}}(-1)^{j}\dim H_{c}^{j}(\hat{\Sigma}_{\hat{w}}'')_{\theta^{-1},\theta'}\text{ is equal to }1\text{ if }F(w)=w\label{eq:Big-thm (b)}\\
 & \text{ and }\text{Ad}(\hat{w}):T'^{F}\rightarrow T^{F}\text{ carries }\theta\text{ to }\theta',\text{ and equals }0\text{ otherwise}.\nonumber \\
\nonumber \\ & H_{c}^{j}(\hat{\Sigma}_{\hat{w}}')_{\theta^{-1},\theta'}=0\text{,\quad for all }j.\label{eq:Big-thm (c)}\end{align}
We first prove \eqref{eq:Big-thm (c)}. For $r=1$ we have $\hat{\Sigma}_{\hat{w}}'=\emptyset$,
so in this case \eqref{eq:Big-thm (c)} is clear. Suppose now that
$r\geq2$. If $M$ is a $\mathcal{T}'^{F}$-module we shall write
$M_{(\chi)}$ for the subspace of $M$ on which $\mathcal{T}'^{F}$
acts according to $\chi$, that is, $M_{(\chi)}=\{m\in M\mid t'm=\chi(t')m,\ \forall t'\in\mathcal{T}'^{F}\}$.
Now $\mathcal{T}'^{F}$ acts on $\hat{\Sigma}_{\hat{w}}'$ by\[
t':(x,x',u,u',z,\tau')\longmapsto(x,t'x't'^{-1},u,t'u't'^{-1},z,\tau't'^{-1}).\]
Hence $H_{c}^{j}(\hat{\Sigma}_{\hat{w}}')$ becomes a $\mathcal{T}'^{F}$-module.
Since $H_{c}^{j}(\hat{\Sigma}_{\hat{w}}')=\bigoplus_{\chi}H_{c}^{j}(\hat{\Sigma}_{\hat{w}}')_{(\chi)}$,
it is enough to show that $H_{c}^{j}(\hat{\Sigma}_{\hat{w}}')_{(\chi)}=0$.
We shall use the definitions and results in Lemmas \ref{lem:roots-commutator},
\ref{lem: induction-commutator}, and the partition \eqref{eq:(*)}
in the end of Section~\ref{sec:Lemmas}, relative to $U,U^{-},V,V^{-}$,
where we take $U,U^{-}$ as above, and $V=\hat{w}(U')^{-}\hat{w}^{-1}$,
$V^{-}=\hat{w}U'\hat{w}^{-1}$. The partition \eqref{eq:(*)} gives
rise to a partition\[
\hat{\Sigma}_{\hat{w}}'=\bigsqcup_{\substack{1\leq a\leq r-1\\
I\in\mathcal{X}}
}\hat{\Sigma}_{\hat{w}}^{a,I},\text{ where }\hat{\Sigma}_{\hat{w}}^{a,I}=\{(x,x',u,u',z,\tau')\in\hat{\Sigma}_{\hat{w}}'\mid z\in Z^{a,*,I}\}.\]
It is easy to see that there is a total order on the set of indices
$(a,I)$ such that the union of the $\hat{\Sigma}_{\hat{w}}^{a,I}$
for $(a,I)$ less than or equal to some given $(a^{0},I^{0})$, is
closed in $\hat{\Sigma}_{\hat{w}}'$. Since the subsets $\hat{\Sigma}_{\hat{w}}^{a,I}$
are stable under the action of $\mathcal{T}'^{F}$, we see that in
order to prove \eqref{eq:Big-thm (c)}, it is enough to show that\begin{equation}
H_{c}^{j}(\hat{\Sigma}_{\hat{w}}^{a,I})_{(\chi)}=0,\quad\mbox{for any fixed }(a,I).\label{eq:Big-thm (d)}\end{equation}
We choose $\alpha'\in I$, and let $\alpha=-\alpha'$. Then $U_{\alpha}\subseteq U\cap V^{-}=U\cap\hat{w}U'\hat{w}^{-1}$.

For any $z\in Z^{a,*}$ and $\xi\in(U_{\alpha})^{r-a-1}$ we have
\[
[\xi,z]=\tau_{\xi,z}\omega_{\xi,z},\]
where $\tau_{\xi,z}\in\mathcal{T}^{\alpha}$ and $\omega_{\xi,z}\in\hat{w}(U')^{r-1}\hat{w}^{-1}$
are uniquely determined (cf.~Lemma~\ref{lem: induction-commutator}).
Moreover, the map $(U_{\alpha})^{r-a-1}\rightarrow\mathcal{T}^{\alpha}$,
$\xi\mapsto\tau_{\xi,z}$ factors through an isomorphism \[
\lambda_{z}:(U_{\alpha})^{r-a-1}/(U_{\alpha})^{r-a}\longiso\mathcal{T}^{\alpha}.\]
Let $\pi:(U_{\alpha})^{r-a-1}\rightarrow(U_{\alpha})^{r-a-1}/(U_{\alpha})^{r-a}$
be the canonical homomorphism. Since $\mathbf{U}_{\alpha}$ is an
affine space, there exists a morphism of algebraic varieties\[
\psi:(U_{\alpha})^{r-a-1}/(U_{\alpha})^{r-a}\longrightarrow(U_{\alpha})^{r-a-1}\]
 such that $\pi\circ\psi=\text{Id}$ and $\psi(1)=1$. Let\[
\mathcal{H}'=\{t'\in\mathcal{T}'\mid t'^{-1}F(t')\in\hat{w}^{-1}\mathcal{T}^{\alpha}\hat{w}\}.\]
This is a closed subgroup of $\mathcal{T}'$. For any $t'\in\mathcal{H}'$
we define $f_{t'}:\hat{\Sigma}_{\hat{w}}^{a,I}\rightarrow\hat{\Sigma}_{\hat{w}}^{a,I}$
by \[
f_{t'}(x,x',u,u',z,\tau')=(xF(\xi),\hat{x}',u,F(t')^{-1}u'F(t'),z,\tau'F(t')),\]
where\[
\xi=(\psi\lambda_{z^{-1}}^{-1}(\hat{w}F(t')^{-1}t'\hat{w}^{-1}))^{-1}\in(U_{\alpha})^{r-a-1}\subseteq U\cap\hat{w}U'\hat{w}^{-1},\]
and $\hat{x}'\in G$ is determined by the condition that defines the
variety $\hat{\Sigma}_{\hat{w}}^{a,I}$, that is,\[
xF(\xi)F(z)F(\hat{w})F(\tau'F(t'))=uz\hat{w}\tau'F(t')F(t')^{-1}u'F(t')\hat{x}'.\]
In order for this to be well-defined we must check that $\hat{x}'\in F(U')$.
Thus we must show that \[
xF(\xi)F(z)F(\hat{w})F(\tau'F(t'))\in uz\hat{w}\tau'u'F(t')F(U').\]
By Lemma~\ref{lem: induction-commutator} we have \[
\xi z=(z^{-1}\xi^{-1})^{-1}=(\omega_{\xi^{-1},z^{-1}}^{-1}\tau_{\xi^{-1},z^{-1}}^{-1}\xi^{-1}z^{-1})^{-1}=z\xi\tau_{\xi^{-1},z^{-1}}\omega_{\xi^{-1},z^{-1}}.\]
Thus the above condition is equivalent to \[
xF(z)F(\xi)F(\tau_{\xi^{-1},z^{-1}})F(\omega_{\xi^{-1},z^{-1}})F(\hat{w})F(\tau'F(t'))\in uz\hat{w}\tau'u'F(t')F(U').\]
Since $xF(z)=uz\hat{w}\tau'u'x'F(\tau')^{-1}F(\hat{w})^{-1}$, it
is enough to show that \begin{eqnarray*}
 & uz\hat{w}\tau'u'x'F(\tau')^{-1}F(\hat{w})^{-1}F(\xi)F(\tau_{\xi^{-1},z^{-1}})F(\omega_{\xi^{-1},z^{-1}})F(\hat{w})F(\tau'F(t'))\\
 & \in uz\hat{w}\tau'u'F(t')F(U'),\end{eqnarray*}
or that \[
x'F(\tau')^{-1}F(\hat{w})^{-1}F(\xi)F(\tau_{\xi^{-1},z^{-1}})F(\omega_{\xi^{-1},z^{-1}})F(\hat{w})F(\tau'F(t'))\in F(t')F(U').\]
Since $x'\in F(U')$ and $F(\hat{w})^{-1}F(\omega_{\xi^{-1},z^{-1}})F(\hat{w})\in F(U')$,
it is enough to check that\[
F(\tau')^{-1}F(\hat{w})^{-1}F(\xi)F(\tau_{\xi^{-1},z^{-1}})F(\hat{w})F(\tau'F(t'))\in F(t')F(U').\]
Since $F(\hat{w}^{-1})F(\xi)F(\hat{w})\in F(U')$ it is enough to
check that \[
F(\tau')^{-1}F(\hat{w})^{-1}F(\tau_{\xi^{-1},z^{-1}})F(\hat{w})F(\tau'F(t'))\in F(t')F(U'),\]
or that \[
F(\hat{w})^{-1}F(\tau_{\xi^{-1},z^{-1}})F(\hat{w})F(F(t'))\in F(t')F(\tau')F(U')F(\tau')^{-1}=F(t')F(U'),\]
which is equivalent to\[
F(\hat{w})^{-1}F(\tau_{\xi^{-1},z^{-1}})F(\hat{w})F(F(t'))=F(t'),\]
that is, $\hat{w}^{-1}\tau_{\xi^{-1},z^{-1}}\hat{w}=F(t')^{-1}t'$,
or $\lambda_{z^{-1}}(\pi(\xi^{-1}))=\tau_{\xi^{-1},z^{-1}}$, which
holds because of the definitions of the element $\xi$ and the map
$\lambda_{z^{-1}}$.

Thus, $f_{t'}:\hat{\Sigma}_{\hat{w}}^{a,I}\rightarrow\hat{\Sigma}_{\hat{w}}^{a,I}$
is well-defined and has an obvious inverse, so is clearly an isomorphism
for any $t'\in\mathcal{H}'$. Note however that this does not define
an action of the group $\mathcal{H}'$ on $\hat{\Sigma}_{\hat{w}}^{a,I}$,
since $f_{t'_{1}t'_{2}}\neq f_{t_{1}'}\circ f_{t_{2}'}$ in general.
Nevertheless, $f_{t'}$ is in particular a well-defined isomorphism
for any $t'\in\mathcal{H}'^{0}$, where $\mathcal{H}'^{0}$ is the
connected component of $\mathcal{H}'$, and by general principles
(cf.~the proof of Prop.~6.4 in \cite{delignelusztig}), the induced
map $f_{t'}^{*}:H_{c}^{j}(\hat{\Sigma}_{\hat{w}}^{a,I})\rightarrow H_{c}^{j}(\hat{\Sigma}_{\hat{w}}^{a,I})$
is constant when $t'$ varies in $\mathcal{H}'^{0}$. In particular,
it is constant when $t'$ varies in $\mathcal{T}'\cap\mathcal{H}'^{0}$.
Now $\mathcal{T}'^{F}\subseteq\mathcal{H}'$ and for $t'\in\mathcal{T}'^{F}$,
the map $f_{t'}$ coincides with the action of $t'^{-1}$ in the $\mathcal{T}'^{F}$-action
on $\hat{\Sigma}_{\hat{w}}^{a,I}$ (we use that $\psi(1)=1$). We
see that the induced action of $\mathcal{T}'^{F}$ on $ $$H_{c}^{j}(\hat{\Sigma}_{\hat{w}}^{a,I})$
is trivial when restricted to $\mathcal{T}'\cap\mathcal{H}'^{0}$.

Now let $n\geq1$ be an integer such that $F^{n}(\hat{w}^{-1}\mathcal{T}^{\alpha}\hat{w})=\hat{w}^{-1}\mathcal{T}^{\alpha}\hat{w}$.
Then \[
t'\longmapsto t'F(t')F^{2}(t')\cdots F^{n-1}(t')\]
is a well-defined morphism $\hat{w}^{-1}\mathcal{T}^{\alpha}\hat{w}\rightarrow\mathcal{H}'$.
Its image is a connected subgroup of $\mathcal{H}'$, hence contained
in $\mathcal{H}'^{0}$. If $t'\in(\hat{w}^{-1}\mathcal{T}^{\alpha}\hat{w})^{F^{n}}$,
then $N_{F}^{F^{n}}(t')\in\mathcal{T}'^{F}$; thus $N_{F}^{F^{n}}(t')\in\mathcal{T}'^{F}\cap\mathcal{H}'^{0}$.
We see that the action of $N_{F}^{F^{n}}(t')\in\mathcal{T}'^{F}$
on $H_{c}^{j}(\hat{\Sigma}_{\hat{w}}^{a,I})$ is trivial for any $t'\in(\hat{w}^{-1}\mathcal{T}^{\alpha}\hat{w})^{F^{n}}$.

If we assume that $H_{c}^{j}(\hat{\Sigma}_{\hat{w}}^{a,I})_{(\chi)}\neq0$,
it follows that $t'\mapsto\chi(N_{F}^{F^{n}}(t'))$ is the trivial
character of $ $$(\hat{w}^{-1}\mathcal{T}^{\alpha}\hat{w})^{F^{n}}$.
This contradicts our assumption that $\chi$ is regular. Thus \eqref{eq:Big-thm (d)}
holds, and hence \eqref{eq:Big-thm (c)} holds.

We now prove \eqref{eq:Big-thm (b)}. Let \[
\tilde{H}=\{(t,t')\in T\times T'\mid tF(t)^{-1}=F(\hat{w})t'F(t')^{-1}F(\hat{w})^{-1}\}.\]
This is a closed subgroup of $T\times T'$ containing $T^{F}\times T'^{F}$.
Now the action of $T^{F}\times T'^{F}$ on $\hat{\Sigma}_{\hat{w}}''$
extends to an action of $\tilde{H}$ given by the same formula. To
see this, consider $(t,t')\in\tilde{H}$ and $(x,x',u,u',1,\tau')\in\hat{\Sigma}_{\hat{w}}''$.
We must show that \[
(txt^{-1},t'x't'^{-1},tut^{-1},t'u't'^{-1},1,\hat{w}^{-1}t\hat{w}\tau't'^{-1})\in\hat{\Sigma}_{\hat{w}}'',\]
that is,\[
txt^{-1}F(\hat{w})F(\hat{w}^{-1})F(t)F(\hat{w})F(\tau')F(t'^{-1})=tut^{-1}\hat{w}\hat{w}^{-1}t\hat{w}\tau't'u't'^{-1}t'x't'^{-1},\]
or that\[
xt^{-1}F(t)F(\hat{w})F(\tau')F(t'^{-1})=u\hat{w}\tau'u'x't'^{-1},\]
or that\[
xt^{-1}F(t)F(\hat{w})F(\tau')F(t'^{-1})=xF(\hat{w})F(\tau')t'^{-1},\]
or that $t^{-1}F(t)F(\hat{w})F(t'^{-1})=F(\hat{w})t'^{-1}$, which
is clear. Let $T_{*}$ (resp. $T'_{*}$) be the reductive part of
$T$ (resp. $T'$) (thus $T_{*}$ is a torus isomorphic to $T$).
Let $\tilde{H}_{*}=\tilde{H}\cap(T_{*}\times T_{*}')$. Then $\tilde{H}_{*}^{0}$
is a torus acting on $\hat{\Sigma}_{\hat{w}}''$ by restriction of
the $\tilde{H}$-action. The fixed point set $(\hat{\Sigma}_{\hat{w}}'')^{\tilde{H}_{*}^{0}}$
is stable under the action of $T^{F}\times T'^{F}$, and by \cite{dignemichel},
4.5 (compare 11.2) and 10.15 we have\begin{multline*}
\sum_{j\in\mathbb{Z}}(-1)^{j}\dim H_{c}^{j}(\hat{\Sigma}_{\hat{w}}'')_{\theta^{-1},\theta'}\\
=|T^{F}\times T'^{F}|^{-1}\sum_{(t,t')\in T^{F}\times T'^{F}}\mathcal{L}((t,t'),\hat{\Sigma}_{\hat{w}}'')\theta(t)\theta'(t')^{-1}\\
=|T^{F}\times T'^{F}|^{-1}\sum_{(t,t')\in T^{F}\times T'^{F}}\mathcal{L}((t,t'),(\hat{\Sigma}_{\hat{w}}'')^{\tilde{H}_{*}^{0}})\theta(t)\theta'(t')^{-1}\\
=\sum_{j\in\mathbb{Z}}(-1)^{j}\dim H_{c}^{j}((\hat{\Sigma}_{\hat{w}}'')^{\tilde{H}_{*}^{0}})_{\theta^{-1},\theta'}.\end{multline*}
It is then enough to show that\begin{align}
 & \sum_{j\in\mathbb{Z}}(-1)^{j}\dim H_{c}^{j}((\hat{\Sigma}_{\hat{w}}'')^{\tilde{H}_{*}^{0}})_{\theta^{-1},\theta'}\text{ is equal to }1\text{ if }F(w)=w\label{eq:Big-thm (e)}\\
 & \text{ and }\text{Ad}(\hat{w}):T'^{F}\rightarrow T^{F}\text{ carries }\theta\text{ to }\theta',\text{ and equals }0\text{ otherwise}.\nonumber \end{align}
Let $(x,x',u,u',1,\tau')\in(\hat{\Sigma}_{\hat{w}}'')^{\tilde{H}_{*}^{0}}$.
By the Lang-Steinberg theorem the first projection $\tilde{H}_{*}\rightarrow T_{*}$
is surjective. It follows that the first projection $\tilde{H}_{*}^{0}\rightarrow T_{*}$
is surjective. Similarly the second projection $\tilde{H}_{*}^{0}\rightarrow T_{*}'$.
Hence for any $t\in T_{*}$, $t'\in T_{*}$ we have\[
txt^{-1}=x,\, t'x't'^{-1}=x',\, tut^{-1}=u,\, t'u't'^{-1}=u',\]
and hence $x=x'=u=u'=1$. Thus $(\hat{\Sigma}_{\hat{w}}'')^{\tilde{H}_{*}^{0}}$
is contained in\begin{equation}
\{(1,1,1,1,1,\tau')\mid\tau'\in T',F(\hat{w}\tau')=\hat{w}\tau'\}.\label{eq:Big-thm (f)}\end{equation}
The set \eqref{eq:Big-thm (f)} is clearly contained in the fixed
point set of $\tilde{H}$. Note that \eqref{eq:Big-thm (f)} is empty
unless $F(w)=w$, by Bruhat decomposition in $G_{1}$. If \eqref{eq:Big-thm (f)}
is empty, then $\sum_{j\in\mathbb{Z}}(-1)^{j}\dim H_{c}^{j}((\hat{\Sigma}_{\hat{w}}'')^{\tilde{H}_{*}^{0}})_{\theta^{-1},\theta'}=0$.
We can therefore assume that $F(w)=w$. Now \eqref{eq:Big-thm (f)}
is stable under the action of $\tilde{H}$. Indeed, if $\tau'\in T'$
is such that $F(\hat{w}\tau')=\hat{w}\tau'$ and $(t,t')\in\tilde{H}$,
then \begin{multline*}
F(\hat{w}\hat{w}^{-1}t\hat{w}\tau't'^{-})=F(\hat{w})F(t')t'^{-1}F(\hat{w})^{-1}tF(\hat{w})F(\tau')F(t'^{-1})\\
=tF(\hat{w})F(t')t'^{-1}F(\tau')F(t'^{-1})=tF(\hat{w})F(\tau')t'^{-1}=\hat{w}\hat{w}^{-1}t\hat{w}\tau't'^{-1}.\end{multline*}
Thus in particular, \eqref{eq:Big-thm (f)} is stable under $\tilde{H}_{*}^{0}$.
Since \eqref{eq:Big-thm (f)} is a finite set and $\tilde{H}_{*}^{0}$
is connected, we see that $\tilde{H}_{*}^{0}$ must act trivially
on \eqref{eq:Big-thm (f)}. Thus, \eqref{eq:Big-thm (f)} is exactly
the fixed point set of $\tilde{H}_{*}^{0}$, and hence $(\hat{\Sigma}_{\hat{w}}'')^{\tilde{H}_{*}^{0}}\cong(\hat{w}T')^{F}$.
Since\begin{multline*}
\#((\hat{\Sigma}_{\hat{w}}'')^{\tilde{H}_{*}^{0}})^{(t,t')}=\#\{\hat{w}\tau'\in(\hat{w}T')^{F}\mid\hat{w}^{-1}t\hat{w}\tau't'^{-1}=\tau'\}\\
=\#\{\hat{w}\tau'\in(\hat{w}T')^{F}\mid\hat{w}^{-1}t\hat{w}=t'\}=\begin{cases}
|(\hat{w}T')^{F}|=|T'^{F}| & \text{if }\hat{w}^{-1}t\hat{w}=t',\\
0 & \text{otherwise,}\end{cases}\end{multline*}
 it follows from facts quoted above together with \cite{lusztigexpo},
1.10 that \begin{multline*}
\sum_{j\in\mathbb{Z}}(-1)^{j}\dim H_{c}^{j}((\hat{\Sigma}_{\hat{w}}'')^{\tilde{H}_{*}^{0}})_{\theta^{-1},\theta'}\\
=|T^{F}\times T'^{F}|^{-1}\sum_{(t,t')\in T^{F}\times T'^{F}}\mathcal{L}((t,t'),(\hat{\Sigma}_{\hat{w}}'')^{\tilde{H}_{*}^{0}})\theta(t)\theta'(t')^{-1}\\
=|T^{F}\times T'^{F}|^{-1}\sum_{t\in T{}^{F}}|T'^{F}|\theta(t)\theta'(\hat{w}^{-1}t\hat{w})^{-1}\\
=\langle\theta,{}^{\hat{w}^{-1}}\theta'\rangle_{T^{F}}=\begin{cases}
1 & \text{if }{}^{\hat{w}}\theta=\theta',\\
0 & \text{otherwise.}\end{cases}\end{multline*}
Thus we have established \eqref{eq:Big-thm (e)}, and so the theorem
is proved.
\end{proof}
We finish by giving some important consequences of the preceding results.
Let $\mathcal{R}(G^{F})$ be the group of virtual representations
of $G^{F}$ over $\overline{\mathbb{Q}}_{l}$. Let $\langle\cdot,\cdot\rangle$
be the standard inner product $\mathcal{R}(G^{F})\times\mathcal{R}(G^{F})\rightarrow\mathbb{Z}$.
Let \[
S_{T,U}=\{g\in G\mid g^{-1}F(g)\in F(U)\}.\]
The finite group $G^{F}\times T^{F}$ acts on $S_{T,U}$ by $(g_{1},t):g\mapsto g_{1}gt^{-1}$.
For any $i\in\mathbb{Z}$ we have an induced action of $G^{F}\times T^{F}$
on $H_{c}^{i}(S_{T,U})$. For $\theta\in\widehat{T^{F}}$ we denote
by $H_{c}^{i}(S_{T,U})_{\theta}$ the subspace of $H_{c}^{i}(S_{T,U})$
on which $T^{F}$ acts according to $\theta$. This is a $G^{F}$-submodule
of $H_{c}^{i}(S_{T,U})$. Let\[
R_{T,U}^{\theta}=\sum_{i\in\mathbb{Z}}(-1)^{i}H_{c}^{i}(S_{T,U})_{\theta}\in\mathcal{R}(G^{F}).\]

\begin{prop}
\label{pro:Geom conj}Let notation be as before. Then the following
holds:
\begin{enumerate}
\item Assume that there exists integers $i$ and $i$' and an irreducible
$G^{F}$-module that appears in the $G^{F}$-module $(H_{c}^{i}(S_{T,U})_{\theta^{-1}})^{*}$
(dual of $H_{c}^{i}(S_{T,U})_{\theta^{-1}}$) and in the $G^{F}$-module
$H_{c}^{i'}(S_{T',U'})_{\theta'}$. Then there exists $n\geq1$ and
$g\in N(T',T)^{F^{n}}$ such that $\mathrm{Ad}(g)$ carries $\theta\circ N_{F}^{F^{n}}|_{\mathcal{T}^{F^{n}}}\in\widehat{\mathcal{T}^{F^{n}}}$
to $\theta'\circ N_{F}^{F^{n}}|_{\mathcal{T}'^{F^{n}}}\in\widehat{\mathcal{T}'^{F^{n}}}$.
\item Assume that there exists an irreducible $G^{F}$-module that appears
in the virtual $G^{F}$-module $R_{T,U}^{\theta}$ and in the virtual
$G^{F}$-module $R_{T',U'}^{\theta'}$. Then there exists $n\geq1$
and $g\in N(T,T')^{F^{n}}$ such that $\mathrm{Ad}(g)$ carries $\theta\circ N_{F}^{F^{n}}|_{\mathcal{T}^{F^{n}}}\in\widehat{\mathcal{T}^{F^{n}}}$
to $\theta'\circ N_{F}^{F^{n}}|_{\mathcal{T}'^{F^{n}}}\in\widehat{\mathcal{T}'^{F^{n}}}$.
\end{enumerate}
\end{prop}
\begin{proof}
We prove (a). Consider the free $G^{F}$-action on $S_{T,U}\times S_{T',U'}$
given by $g_{1}:(g,g')\mapsto(g_{1}g,g_{1}g')$. The map\[
(g,g')\longmapsto(x,x',y),\quad x=g^{-1}F(g),x'=g'^{-1}F(g'),y=g^{-1}g'\]
defines an isomorphism from $G^{F}\backslash(S_{T,U}\times S_{T',U'})$
to $\Sigma$ (the fact that it is an isomorphism and not merely a
bijective homomorphism is proved in \cite{carter}, p.~221-222 in
the situation where $r=1$; the same argument works in general). The
action of $T^{F}\times T'^{F}$ on $S_{T,U}\times S_{T',U'}$ given
by right multiplication by $t^{-1}$ on the first factor and by $t'^{-1}$
on the second factor, becomes an action of $T^{F}\times T'^{F}$ on
$\Sigma$ given by $(x,x',y)\mapsto(txt^{-1},t'x't'^{-1},tyt'^{-1})$.
Our assumption implies that the $G^{F}$-module $H_{c}^{i}(S_{T,U})_{\theta^{-1}}\otimes H_{c}^{i'}(S_{T',U'})_{\theta'}$
contains the trivial representation with non-zero multiplicity, that
is, $(H_{c}^{i}(S_{T,U})_{\theta^{-1}}\otimes H_{c}^{i'}(S_{T',U'})_{\theta'})^{G^{F}}\neq0$.
By \cite{dignemichel}, 10.9 and 10.10(i) we have an inclusion \[
(H_{c}^{i}(S_{T,U})_{\theta^{-1}}\otimes H_{c}^{i'}(S_{T',U'})_{\theta'})^{G^{F}}\longhookrightarrow H_{c}^{i+i'}(G^{F}\backslash(S_{T,U}\times S_{T',U'}))_{\theta^{-1},\theta'},\]
and so $H_{c}^{i+i'}(G^{F}\backslash(S_{T,U}\times S_{T',U'}))_{\theta^{-1},\theta'}\neq0$.
By the above isomorphism we thus have $H_{c}^{i+i}(\Sigma)_{\theta^{-1},\theta'}\neq0$.
We now use Lemma~\ref{lem:Lusztig's 1.4} and (a) follows.

We prove (b). By \cite{dignemichel}, 11.4 we have\[
\sum_{i}(-1)^{i}(H_{c}^{i}(S_{T,U})_{\theta^{-1}})^{*}=\sum_{i}(-1)^{i}H_{c}^{i}(S_{T,U})_{\theta}.\]
Hence the assumption of (b) implies that the assumption of (a) holds.
Hence the conclusion of (a) holds. The proposition is proved.
\end{proof}
\begin{prop}
\label{pro:scalar product formula}Assume that $\theta$ or $\theta'$
is regular (see Definition~\ref{def:regular}). Then\[
\langle R_{T,U}^{\theta},R_{T',U'}^{\theta'}\rangle=\#\{w\in W(T_{1},T_{1}')^{F}\mid{}^{\hat{w}}\theta=\theta'\}.\]

\end{prop}
\begin{proof}
We may assume that $\theta'$ is regular. We have \begin{multline*}
\langle R_{T,U}^{\theta},R_{T',U'}^{\theta'}\rangle\\
=\sum_{i,i'\in\mathbb{Z}}(-1)^{i+i'}\dim(H_{c}^{i}(S_{T,U})_{\theta^{-1}}\otimes H_{c}^{i'}(S_{T',U'})_{\theta'})^{G^{F}}\\
=\sum_{j\in\mathbb{Z}}(-1)^{j}\dim H_{c}^{j}(G^{F}\backslash(S_{T,U}\times S_{T',U'}))_{\theta^{-1},\theta'}\\
=\sum_{j\in\mathbb{Z}}(-1)^{j}\dim H_{c}^{j}(\Sigma)_{\theta^{-1},\theta'}.\end{multline*}
It remains to use Theorem~\ref{thm:Big-thm}.
\end{proof}
\begin{cor}
\label{cor:irrep & indepU}Assume that $\theta\in\widehat{T^{F}}$
is regular. Then
\begin{enumerate}
\item $R_{T,U}^{\theta}$ is independent of the choice of $U$.
\item Assume in addition that the there is no non-trivial element $w\in W(T_{1})^{F}$
such that $\hat{w}$ fixes $\theta$. Then $\pm R_{T,U}^{\theta}$
is an irreducible $G^{F}$-module.
\end{enumerate}
\end{cor}
\begin{proof}
We prove (a). Let $V$ be the subgroup of $G$ associated with the
unipotent radical $\mathbf{V}$ of another Borel subgroup of $\mathbf{G}$
containing $\mathbf{T}$. By Prop.~\ref{pro:scalar product formula}
we have \[
\langle R_{T,U}^{\theta},R_{T,U}^{\theta}\rangle=\langle R_{T,U}^{\theta},R_{T,U'}^{\theta}\rangle=\langle R_{T,U'}^{\theta},R_{T,U}^{\theta}\rangle=\langle R_{T,U'}^{\theta},R_{T,U'}^{\theta}\rangle.\]
Hence $\langle R_{T,U}^{\theta}-R_{T,U'}^{\theta},R_{T,U}^{\theta}-R_{T,U'}^{\theta}\rangle=0$,
and so $R_{T,U}^{\theta}=R_{T,U'}^{\theta}$. This proves (a). In
the setup of (b), we see from Prop.~\ref{pro:scalar product formula}
that $\langle R_{T,U}^{\theta},R_{T,U}^{\theta}\rangle=1$, which
proves (b).
\end{proof}
\bibliographystyle{alex}
\bibliography{alex}

\end{document}